\documentclass[12pt]{article}
\usepackage{geometry} 
\usepackage[ utf8 ]{inputenc}
\usepackage[T1]{fontenc}
\usepackage[english,french]{babel}
\usepackage{amsmath}
\usepackage{amsfonts}
\usepackage{amssymb}
\usepackage{amsthm}
\usepackage{textcomp}
\usepackage{eucal}
\usepackage{array}
\theoremstyle{plain}
\newtheorem{theorem}{Théorème}
\newtheorem{corollary}{Corollaire}
\newtheorem{lemma}{Lemme}
\newtheorem{proposition}{Proposition}

\theoremstyle{definition}

\theoremstyle{remark}
\newtheorem{remark}{Remarque}
\date{}
\title{ Représentations de réflexion de groupes de Coxeter\\Sixième partie: le cas réductible; exemples en rang trois}
\author{François ZARA}
\begin{document}
\maketitle
\begin{abstract}
Dans cette sixième partie nous étudions les groupes de réflexion complexes en rang $3$ non bien engendrés: $G(2r,r,2)$, $G_{12}$, $G_{13}$ et $G_{22}$. Nous partons d'une représentation de réflexion d'un groupe de Coxeter de rang $3$ et nous montrons que nous pouvons ainsi obtenir des groupes de réflexion affines dont la partie linéaire est l'un des groupes précédents.
\end{abstract}
\begin{otherlanguage}{english}
\begin{abstract}
In this sixth part we study rank $3$ reflection groups not well generated: $G(2r,r,2)$, $G_{12}$, $G_{13}$ and $G_{22}$. We start from a reflection representation of a rank $3$ Coxeter group and we show that we can obtain in this manner affine Coxeter groups whose linear part is one of the preceding groups.
\end{abstract}
\end{otherlanguage}
\let\thefootnote\relax\footnote{Mots clés et phrases: groupes de Coxeter, groupes de réflexion.Représentation de réflexion réductible.}
\let\thefootnote\relax
\section{Introduction}
Des présentations des groupes de réflexion complexes ont été donnés dans \cite{BMR} et \cite{BR} et Shi dans \cite{S} a, en partant de celles de \cite{BR}, donné une autre présentation du groupe $G_{22}$. Ici notre méthode est complètement différente. On veut trouver des présentations des groupes $G(2r,r,2)$, $G_{12}$, $G_{13}$ et $G_{22}$ comme quotients de groupes de Coxeter de rang $3$ $W(p,q,r)$: on leur applique la construction fondamentale. La structure de ces groupes permet de trouver des restrictions sur les entiers $p$, $q$ et $r$; puis nous exprimons que la représentation est réductible. En exprimant que nous sommes sur le bon corps, nous trouvons toutes les possibilités. notre méthode donne aussi des extensions affines de ces groupes.
\section{Les groupes de réflexion complexe imprimitifs $G(2r,r,2)$ $(r\geqslant 3)$}
\subsection{Le groupe de Coxeter $W(4,4,r)$}
\begin{proposition}
Soient $r$ un entier $\geqslant 3$ et $\gamma$ une racine de $v_{r}(X)$. On considère le groupe de Coxeter $W(4,4,r)$ et l'une de ses représentations affines $R$. On pose $G:=Im R=<s_{1},s_{2},s_{3}>$. Alors:
\begin{enumerate}
  \item  $l$ et $m$ sont les racines du polynôme $Q(X)=X^{2}+\gamma X+\gamma$.\\
  On suppose que $\gamma=4\cos^{2}\frac{k\pi}{r}$ avec $(k,r)=1$ et $0< k < \frac{r}{2}$ alors
  \begin{eqnarray*}
 \{l,m\} & = & \{\frac{-1}{2}(\gamma+i\sqrt{\gamma(4-\gamma))},\frac{-1}{2}(\gamma-i\sqrt{\gamma(4-\gamma)})\}\\
  & = & \{-1-\exp (\frac{2k\pi i}{r}),-1-\exp (\frac{-2k\pi i}{r})\}.
\end{eqnarray*}
Le choix de $l$ détermine une représentation de réflexion affine de $W(4,4,r)$, l'autre choix correspond à la représentation duale $R^{\star}$. Le système de paramètres de $G$ est $\mathcal{P}(G)=\mathcal{P}(2,2,\gamma;2l,2m)$.
 \item La représentation $R$ est réalisée sur le corps $K=\mathbb{Q}(\zeta_{r})$ où $\zeta_{r}$ est une racine primitive $r$-ième de l'unité. On a $K_{0}=\mathbb{Q}(\zeta_{r}+\zeta_{r}^{-1})$ et $K_{0}$ est le sous-corps réel de $K$. (Si $r=2r_{1}$ avec $r_{1}$ impair , on a $\mathbb{Q}(\zeta_{r})=\mathbb{Q}(\zeta_{r_{1}})$).
\end{enumerate}
\end{proposition}
\begin{proof}
On a $\Delta=8-4-4-2\gamma-2(l+m)=0$ donc $l+m=-\gamma$ et $lm=\gamma$: $l$ et $m$ sont les racines  du polynôme $Q(X)=X^{2}+\gamma X+\gamma$ et l'on voit facilement que l'on a les valeurs de l'énoncé. Tout le reste est clair.
\end{proof}
Des calculs simples montrent que l'on a les relations suivantes entre $l$ et $m$:
\begin{equation}
4-\gamma=(l+2)(m+2)=(l+2)+(m+2)=4\sin^{2} \frac{k\pi}{r};
\end{equation}
\begin{equation}
1=(l+1)(m+1),\; m+2=(l+2)(m+1),\; l+2=(m+2)(l+1);
\end{equation}
\begin{equation}
(4-\gamma)(l+1)=(l+2)^{2},\; (4-\gamma)(m+1)=(m+2)^{2};
\end{equation}
\begin{equation}
\frac{1}{l+2}+\frac{1}{m+2}=1.
\end{equation}
Donc $l+2$ et $m+2$ sont les racines du polynôme $Q'(X)=X^{2}+(4-\gamma)X+(4-\gamma)$.
\begin{proposition}
On garde les hypothèses et notations précédentes. On pose $z_{3}:=(s_{1}s_{2})^{2}$, $s_{13}:=s_{1}z_{3}$ et $s_{23}:=s_{2}z_{3}$. On a $G=<s_{1},s_{23},s_{3}>$ avec le système de paramètres $\mathcal{P}(2,2,4-\gamma;2(l+2),2(m+2))$. Il en résulte que $G$ est l'image d'une représentation de réflexion affine du groupe de Coxeter $W(4,4,r')$.
\end{proposition}
\begin{proof}
Elle ne présente pas de difficultés.
\end{proof}
\textbf{Si $r$ est impair on a  $r'=2r$; on peut donc supposer que $r$ est pair, ce que nous ferons dans la suite. On pose $r=2r_{1}$}
\begin{proposition}
Le groupe $G$ a comme ``présentation'':
\[
G=<w(4,4,r),(s_{1}s_{3}^{s_{2}})^{4}=1>.
\]
\end{proposition}
\begin{proof}
On a $C(s_{1,}s_{3}^{s_{2}})=\alpha\gamma+\beta+(\alpha l+\beta m)=2\gamma+2(l+m)$ donc $C(s_{1},s_{3}^{s_{2}})=2$ équivaut à $l+m=-\gamma$, ce qui est aussi équivalent à $\Delta=0.$ On a le résultat car $C(s_{1},s_{3}^{s_{2}})=2$ équivaut à $s_{1}s_{3}^{s_{2}}$ d'ordre $4$.
\end{proof}
\subsection{Le groupe $G'=G/N$}
La représentation $R$ sur $M$ induit une représentation $R'$ de $G'$ sur $M'=M/<b>$ qui est de dimension $2$. Comme $SL_{2}(M')$ ne possède qu'un élément d'ordre $2$, nous obtenons les relations $z'=(s'_{1}s'_{2})^{2}=(s'_{1}s'_{3})^{2}=(s'_{2}s'_{3})^{r_{1}}$, $z'\in Z(G')$. (Ici $g'$ désigne l'image de $g$ dans $G'$.) Posons $A:=<s'_{2},s'_{3}>$. Alors $A$ est un groupe diédral isomorphe à $D_{r}$ et $z'\in A$. Nous avons $s'_{1}s'_{2}s'_{1}=z's'_{2}$ et $s'_{1}s'_{3}s'_{1}=z's'_{3}$, donc $A\lhd G'$, $G'=<A,s'_{1}>$, d'où $|G'|=4r$ et $G'\simeq D_{r}\rtimes C_{2}$.

La représentation de $<s'_{2}s'_{3}>$ sur $M'$ n'est pas irréductible. Dans la base $\mathcal{B}$ de $M$ on a 
\[
s_{2}=
\begin{pmatrix}
1 & 0 & 0\\
0 & -1 & l\\
0 & 0 &1
\end{pmatrix}
\text{et}\;
s_{3}=
\begin{pmatrix}
1 & 0 & 0\\
0 & 1 & 0\\
0 & m & -1
\end{pmatrix}
\]
donc dans la base $\mathcal{B}'=(a'_{2},a'_{3})$ de $M'$ ($a'_{i}$ image de  $a_{i}$ dans $M'$), on a 
\[
s'_{2}=
\begin{pmatrix}
-1 & l\\
0 & 1
\end{pmatrix}
,
s'_{3}=
\begin{pmatrix}
1 & 0\\
m & -1
\end{pmatrix}
,
s'_{2}s'_{3}=
\begin{pmatrix}
\gamma-1 & -l\\
m & -1
\end{pmatrix}
.
\]
\[
P_{s'_{2}s'_{3}}(X)=x^{2}-(\gamma-2)X+1=(X+l+1)(X+m+1).
\]
Un vecteur propre correspondant à la valeur propre $-(l+1)$ (resp. $-(m+1)$) est $v_{1}=la'_{2}-ma'_{3}$ (resp. $v_{2}=a'_{2}-a'_{3}$). Posons $V_{i}:=<v_{i}>\,(i\in\{2,3\})$. Alors $M'=V_{1}\oplus V_{2}$, $s'_{2}$ et $s'_{3}$ échangent $V_{1}$ et $V_{2}$ tandis que $s'_{1}$ les stabilise. Ainsi nous voyons que $G'$ induit un groupe de réflexion complexe imprimitif sur $M'$. On a donc $G'\simeq G(r,r/2,2)$ (on rappelle que $r$ est pair).
\begin{proposition}
Une présentation de $G'$ est:
\[
<s_{1},s_{2},s_{3}|s_{i}^{2}=1=(s_{2}s_{3})^{r} (1\leqslant i \leqslant 3), (s_{1}s_{2})^{2}=(s_{1}s_{3})^{2}=(s_{2}s_{3})^{r/2}>
\]
\end{proposition}
\begin{proof}
Dans un groupe $H$ ayant la présentation de l'énoncé on a $(s_{1}s_{2})^{4}=(s_{1}s_{3})^{4}=(s_{2}s_{3})^{r}=1$, donc $H$ est isomorphe à un quotient du groupe $W(4,4,r)$. On construit ainsi facilement un isomorphisme entre $H$ et $G'$ d'où le résultat.
\end{proof}
\subsection{Le sous-groupe $N$ des translations de $G$}
On pose $e_{1}:=(s_{2}s_{3})^{r_{1}}(s_{1}s_{3})^{2}$ et $e_{2}:=(s_{2}s_{3})^{r_{1}}(s_{1}s_{2})^{2}$. Soit $\mathcal{R}$ le sous-anneau de $\mathbb{Z}[\zeta_{r}]$ engendré par $\zeta_{r}+\zeta_{r}^{-1}$ et $2\zeta_{r}$. Le but de cette partie est de démontrer le résultat suivant:
\begin{proposition}
Le groupe $N$ est un $\mathcal{R}$-module libre de rang $2$ de base $(e_{1},e_{2})$. C'est aussi un $\mathbb{Z}$-module libre de rang $2n$.
\end{proposition}
\begin{proof}
Un calcul simple montre que l'on a:
\[
e_{1}=(\frac{2}{l+2},0),e_{2}=(0,\frac{2}{m+2}).
\]
On en déduit aussitôt que:
\begin{eqnarray*}
c_{1} & = & (l+2)e_{1}+(m+2)e_{2} \\
c_{2} & = & -(l+2)e_{1}+\frac{l(m+2)}{2}e_{2}\\
c_{3} & = & \frac{m(l+2)}{2}e_{1}-(m+2)e_{2} 
\end{eqnarray*}
On obtient alors:
\[
s_{1}.e_{1}  =  (1-\frac{2}{m+2})e_{1}-\frac{2}{l+2}e_{2}  =  \frac{m}{m+2}e_{1}-\frac{2}{l+2}e_{2} 
\]
\[
s_{1}.e_{2}  =  -\frac{2}{m+2}e_{1}+(1-\frac{2}{l+2})e_{2}  =  -\frac{2}{m+2}e_{1}+\frac{l}{l+2}e_{2} 
\]
\[
s_{2}.e_{1}=-e_{1}-me_{2},\;s_{2}.e_{2}=e_{2};\qquad s_{3}.e_{1}=e_{1},\;s_{3}.e_{2}=-le_{1}-e_{2}
\]
en utilisant les relations du début et le fait que $l(m+1)=\gamma+l=-m$ et $m(l+1)=\gamma+m=-l$.\\
Dans le système $(e_{1},e_{2})$ de $N$ on a:
\[
s_{1}=
\begin{pmatrix}
\frac{m}{m+2} & \frac{-2}{m+2}\\
\frac{-2}{l+2} & \frac{l}{l+2}
\end{pmatrix}
,
s_{2}=
\begin{pmatrix}
-1 & 0\\
-m & 1
\end{pmatrix}
,
s_{3}=
\begin{pmatrix}
1 & -l\\
0 & -1
\end{pmatrix}
\]
\[
s_{1}^{2}=s_{2}^{2}=s_{3}^{2}=I,\;(s_{1}s_{2})^{2}=(s_{1}s_{3})^{2}=(s_{2}s_{3})^{r_{1}}=-I.
\]
Nous montrons maintenant que les formules précédentes donnant $s_{1}$, $s_{2}$ et $s_{3}$ peuvent s'écrire dans $\mathcal{R}$. Distinguons deux cas suivant que $r$ est une puissance de $2$ ou non.\\
- Si $r$ n'est pas une puissance de $2$, nous savons que $4-\gamma$ est inversible, donc aussi $l+2$ et $m+2$ puisque $(l+2)(m+2)=4-\gamma$. Si $P(X)\in \mathbb{Z}[X]$ est tel que $(4-\gamma)P(\gamma)=1$ alors on a $(l+2)^{-1}=(m+2)P(\gamma)$ et a $(m+2)^{-1}=(l+2)P(\gamma)$ donc:
\begin{eqnarray*}
s_{1}.e_{1} & = & m(l+2)P(\gamma)e_{1}-2(m+2)P(\gamma)e_{2} \\
s_{1}.e_{2} & = & -2(l+2)P(\gamma)e_{1}+l(m+2)P(\gamma)e_{2}
\end{eqnarray*}
- Si $r$ est une puissance de $2$, il existe  $P(X)\in \mathbb{Z}[X]$  tel que $(4-\gamma)P(\gamma)=2$ d'où $(l+2)(m+2)P(\gamma)=2$ donc:
\begin{eqnarray*}
s_{1}.e_{1} & = & (1-(l+2)P(\gamma))e_{1}-(m+2)P(\gamma)e_{2} \\
s_{1}.e_{2} & = & -(l+2)P(\gamma)e_{1}+(1-(m+2)P(\gamma))e_{2}.
\end{eqnarray*}
Dans les deux cas, $r$ puissance de $2$ ou non, $N$ est engendré par $e_{1}$ et $e_{2}$ sur l'anneau $\mathcal{R}$. En effet on a $\gamma=\zeta_{r}+\zeta_{r}^{-1}$ et $\alpha l+2=2l+2=-2\zeta_{r}$ et $\mathcal{R}=\mathbb{Z}[\gamma,-2\zeta_{r}]$.\\
Comme $(e_{1},e_{2})$ est un système libre, $N$ est un $\mathcal{R}$-module libre de rang $2$ de base $(e_{1},e_{2})$. Il est clair que c'est aussi un $\mathbb{Z}$-module libre de rang $2r$.
\end{proof}
\subsection{La suite $(\star\star)$ es non scindée}
\begin{proposition}
On a: $[N,s_{1}]=\,<\frac{2}{4-\gamma}c_{1}>$, $[N,s_{2}]=\,<\frac{2}{l+2}c_{2}>$\\ et $[N,s_{3}]=\,<\frac{2}{m+2}c_{3}>$.
\end{proposition}
\begin{proof}
Comme $N=<e_{1},e_{2}>$, on a $[N,s_{1}]=<[s_{1},e_{1}],[s_{1},e_{2}]>$ et 
\[ [s_{1},e_{1}]=-\frac{2}{m+2}e_{1}-\frac{2}{l+2}e_{2}=-\frac{2}{m+2}(\frac{2}{l+2},0)-\frac{2}{l+2}(0,\frac{2}{m+2})=\frac{2}{4-\gamma}c_{1},
\]
\[
[s_{1},e_{2}]=-\frac{2}{m+2}e_{1}-\frac{2}{l+2}e_{2}=\frac{2}{4-\gamma}c_{1}, 
\]
donc $[N,s_{1}]=<\frac{2}{4-\gamma}c_{1}>$. On montre de même les deux autres résultats.
\end{proof}
\begin{proposition}
La suite $(\star\star)$ es non scindée.
\end{proposition}
\begin{proof}
Nous utilisons les résultats des corollaires 1 et 2 du chapitre 8. Soient $\lambda_{i} \,(1\leqslant i\leqslant3)$ des scalaires tels que $\lambda_{i}c_{i}\in N,\,\lambda_{i}\in K^{\star}$ et l'on pose $\sigma(s_{i})=s_{i}'=s_{i}(\lambda_{i}c_{i})$. On cherche les $\lambda_{i}$ de telle sorte que l'on ait les relations:
\[
s_{i}'^{2}=1\,(1\leqslant i\leqslant3),\, (s_{2}'s_{3}')^{r}=1,\, (s_{1}'s_{2}')^{2}=(s_{1}'s_{3}')^{2}=(s_{2}'s_{3}')^{\frac{r}{2}}.
\]
On appelle $z_{i}$ l'unique élément de $\mathcal{Z}'$ qui centralise $s_{j}$ et $s_{k}$ si $|\{i,j,k\}|=3$. Il suffit de vérifier (si possible) les relations:
\[
(s_{1}'s_{2}')^{2}=(s_{1}'s_{3}')^{2}=(s_{2}'s_{3}')^{r_{1}}
\]
où $r=2r_{1}$.\\
On a $(s_{1}'s_{2}')^{2}=z_{3}b_{2}$, $(s_{1}'s_{3}')^{2}=z_{2}b_{2}'$ et $(s_{2}'s_{3}')^{r_{1}}=z_{1}b_{r_{1}}''$. Nous avons, en rempla\c cant chaque $c_{i}$ par sa valeur: $c_{1}=(2,2)$, $c_{2}=(-2,l)$ et $c_{3}=(m,-2)$:\\
$b_{2}=(-2\lambda_{2},2(l+2)\lambda_{1}+2(l+1)\lambda_{2})$, $b_{2}'=(2(m+2)\lambda_{1}+2(m+1)\lambda_{3},-2\lambda_{3})$, $b_{r_{1}}''=-2(\lambda_{2},\lambda_{3})$.\\
Nous avons $z_{1}z_{3}+b_{2}'=b_{r_{1}}''$ 
\[
z_{1}z_{3}=(0,\frac{2}{m+2})=b_{r_{1}}''-b_{2}'=(\star,0)
\]
ce qui est impossible car $\frac{2}{m+2}\neq 0$. On ne peut donc pas trouver les scalaires $\lambda_{i}$ et la suite $(\star\star)$ est non scindée. (On rappelle que l'on emploie la notation additive pour $N$.)
\end{proof}
\section{Le groupe de réflexion complexe $G_{12}$ }
\subsection{Généralités}
On connait la structure du groupe de réflexion complexe $G_{12}$ ($\simeq Q\rtimes S_{3}$ où $Q$ est le groupe des quaternions). On sait aussi que le corps de définition de $G_{12}$ en tant que groupe de réflexion complexe est $K=\mathbb{Q}(\sqrt{-2})$. Comme $K_{0}$ est un sous-corps réel de $K$, on a $K_{0}=\mathbb{Q}$.

Le groupe $G_{12}$ n'est pas bien engendré: il est de rang $2$ mais nécessite trois réflexions comme générateurs. Nous allons partir de ces trois réflexions et appeler $G$ le groupe ainsi obtenu dans un espace à $3$ dimensions. On suppose en plus que $\Delta(G)=0$. Il résulte de ces considérations que $G$ va être quotient d'un groupe de Coxeter $W(p,q,r)$ où $\{p,q,r\}\subset\{3,4,6\}$ ($2$ n'intervient pas car le groupe est un groupe de réflexion complexe.)
\subsection{Recherche des triples $(p,q,r)$}
Nous cherchons les triples $(p,q,r)$ (à l'ordre près) pour lesquels $K=\mathbb{Q}(\sqrt{-2})$.
\begin{remark}
\begin{enumerate}
  \item Les triples $(4,4,3)$, $(4,4,4)$ et $(4,4,6)$ donnent des groupes de réflexion complexes imprimitifs. Ile ont été étudiés au paragraphe 1.
  \item Les triples $(3,3,3)$ et$(6,6,3)$ donnent un groupe diédral affine. Ils ont été étudiés au chapitre 6.
 \end{enumerate}
\end{remark}
\begin{proposition}
On garde les hypothèses et notations précédentes. On obtient le corps $K=\mathbb{Q}(\sqrt{-2})$ pour les cinq triples suivants: (on appelle $A_{i}$ le groupe correspondant)
\begin{enumerate}
  \item $(3,3,4)$. $\mathcal{P}(A_{1})=\mathcal{P}(1,1,2;\sqrt{-2},-\sqrt{-2})$. On pose $t=s_{1}s_{2}s_{3}$ et $t'$ l'image de $t$ dans $A_{1}'$: $P_{t'}(X)=X^{2}+\sqrt{-2}X-1$ :$t'^{4}=-1$
  \item $(3,3,6)$. $\mathcal{P}(A_{2})=\mathcal{P}(1,1,3;-1+\sqrt{-2},-1-\sqrt{-2})$. On pose $t=s_{1}s_{2}s_{3}$ et $t'$ l'image de $t$ dans $A_{1}'$: $P_{t'}(X)=X^{2}+\sqrt{-2}X-1$ :$t'^{4}=-1$
  \item $(3,4,6)$. $\mathcal{P}(A_{3})=\mathcal{P}(1,2,3;-2+\sqrt{-2},-2-\sqrt{-2})$. On pose $t=s_{1}s_{2}s_{3}$ et $t'$ l'image de $t$ dans $A_{1}'$: $P_{t'}(X)=X^{2}+\sqrt{-2}X-1$ :$t'^{4}=-1$. 
 \item $(6,6,4)$. $\mathcal{P}(A_{4})=\mathcal{P}(3,3,2;-4+\sqrt{-2},-4-\sqrt{-2})$. On pose $t=s_{1}s_{2}s_{3}$ et $t'$ l'image de $t$ dans $A_{1}'$: $P_{t'}(X)=X^{2}+\sqrt{-2}X-1$ :$t'^{4}=-1$.
 \item $(6,6,6)$. $\mathcal{P}(A_{5})=\mathcal{P}(1,1,2;-5+\sqrt{-2},-5-\sqrt{-2})$. On pose $t=s_{1}s_{2}s_{3}$ et $t'$ l'image de $t$ dans $A_{1}'$: $P_{t'}(X)=X^{2}+\sqrt{-2}X-1$ :$t'^{4}=-1$
\end{enumerate}
Ces cinq groupes sont isomorphes à un groupe noté $G$.
\end{proposition}
\begin{proof}
Les calculs pour les valeurs des paramètres n'offrent pas de difficultés. Nous montrons maintenant que ces cinq groupes sont isomorphes.\\
1) Soit $A_{1}$. Nous avons $\alpha=\beta=1$, $\gamma=2$ et $\alpha l+\beta m=0$ donc $C(s_{1},s_{2}(s_{2}s_{3})^{2})=4-\alpha=3=4-\beta =C(s_{1},s_{3}(s_{2}s_{3})^{2})$. Comme $<s_{2},s_{3}>=<s_{2}(s_{2}s_{3})^{2},s_{3}(s_{2}s_{3})^{2}>$ nous obtenons $A_{1}=<s_{1},s_{2}(s_{2}s_{3})^{2},s_{3}(s_{2}s_{3})^{2}>$ et $A_{1}$ est un quotient de $W(6,6,4)$. Nous voyons donc que groupe $A_{4}$ est isomorphe au groupe $A_{1}$.\\
2) Soit $A_{4}$. Nous avons $\alpha=\beta=3$, $\gamma=2$ et $\alpha l+\beta m=-4$ donc $C(s_{1},s_{2}(s_{2}s_{3})^{3})=4-\alpha=1$, $C(s_{2},s_{3}(s_{1}s_{3})^{3})=4-\gamma=2$ et $A_{4}$ est un quotient de $W(3,6,4)$. Le groupe $A_{4}$ est isomorphe au groupe $A_{3}$.\\
3)Soit $A_{2}$. Nous avons $\alpha=\beta=1$, $\gamma=3$ et $\alpha l+\beta m=-2$ donc $C(s_{1},s_{2}(s_{2}s_{3})^{3})=3=C(s_{1},s_{3}(s_{1}s_{3})^{3})$ et nous voyons ainsi que $A_{2}$ est isomorphe à $A_{5}$. Nous avons $C(s_{1},s_{2}^{s_{3}})=1+3-2=2$, $<s_{2},s_{3}>=<_{2}^{s_{3}},s_{3}>$ donc $A_{2}$ est un quotient de $W(4,6,3)$.Le groupe $A_{2}$ est isomorphe au groupe $A_{3}$.
\end{proof}
\begin{proposition}
Des ``présentations'' du groupe $G$ sont les suivantes (avec les notations précédentes):
\begin{enumerate}
  \item $A_{1}$ . $(w(3,3,4),(s_{1}s_{2}^{s_{3}})^{6}=1).$
  \item $A_{2}$ . $(w(3,3,6),(s_{1}s_{2}^{s_{3}})^{4}=1).$
  \item $A_{3}$ . $(w(3,4,6),(s_{1}s_{2}^{s_{3}})^{6}=1).$
  \item $A_{4}$ . $(w(6,6,4),(s_{1}s_{2}^{s_{3}})^{3}=1).$
  \item $A_{5}$ . $(w(6,6,6),(s_{1}s_{2}^{s_{3}})^{4}=1).$
\end{enumerate}
\end{proposition}
\begin{proof}
Comme $C(s_{1},s_{2}^{s_{3}})=\alpha+\beta\gamma+\alpha l+\beta m$ on a les équivalences:
\begin{itemize}
  \item $A_{1}$ . $w(3,3,4)$, $C(s_{1},s_{2}^{s_{3}})=3\Leftrightarrow l+m=0 \Leftrightarrow \Delta(G)=0$.
  \item $A_{2}$ . $w(3,3,6)$, $C(s_{1},s_{2}^{s_{3}})=2\Leftrightarrow l+m=-2 \Leftrightarrow \Delta(G)=0$.
  \item $A_{3}$ . $w(3,4,6)$, $C(s_{1},s_{2}^{s_{3}})=3\Leftrightarrow l+2m=-4 \Leftrightarrow \Delta(G)=0$.
  \item $A_{4}$ . $w(6,6,4)$, $C(s_{1},s_{2}^{s_{3}})=1\Leftrightarrow 3l+3m=-8 \Leftrightarrow \Delta(G)=0$.
  \item $A_{5}$ . $w(6,6,6)$, $C(s_{1},s_{2}^{s_{3}})=2\Leftrightarrow 6l+6m=-10 \Leftrightarrow \Delta(G)=0$.
\end{itemize}
d'où le résultat car il n'y a qu'une valeur possible pour chaqun de $\alpha$, $\beta$ et $\gamma$.
\end{proof}
\subsection{Le groupe $G'$ est isomorphe au groupe $G_{12}$}
\begin{theorem}
On garde les hypothèses et notations de la proposition 8. Alors $G'=G/N$ est isomorphe au groupe de réflexion $G_{12}$.
\end{theorem}
\begin{proof}
Nous montrons le résultat pour le groupe $A_{1}$. Posons $z:=(s_{2}'s_{3}')^{2}$. Comme $s_{1}'z\in N(A_{1})$ et comme $z$ est d'ordre $2$, nous voyons que $z$ est central dans $A_{1}'$. Maintenant $s_{2}'s_{3}'$ est d'ordre $4$ et $(s_{1}'(s_{2}'s_{3}')s_{1}')^{2}= (s_{2}'s_{3}')^{2}$\\ donc $Q:=<(s_{2}'s_{3}'),(s_{1}'(s_{2}'s_{3}')s_{1}')>$ est un groupe de quaternion d'ordre $8$. Le groupe $Q$ est normalisé par $s_{1}'$ et aussi par $s_{2}'$ puisque 
\[
s_{2}'s_{1}'(s_{2}'s_{3}')s_{1}'s_{2}'=(s_{2}'s_{1}'s_{2}')s_{3}'s_{1}'s_{2}'=s_{1}'s_{2}'s_{1}'s_{3}'s_{1}'s_{2}'=s_{1}'s_{2}'s_{3}'s_{1}'s_{3}'s_{2}'\in Q.
\]
Comme$s_{2}'s_{3}'\in Q$, $s_{3}'$ normalise $Q$ et $Q\lhd A_{1}'$. Il est clair que $A_{1}'=Q<s_{1}',s_{2}'>$ et $Q\bigcap<s_{1}',s_{2}'>=\{1\}$ donc $A_{1}'\simeq Q \rtimes S_{3}\simeq G_{12}$.
\end{proof}
\begin{proposition}
Des présentations du groupe $G_{12}$ sont les suivantes (avec les notations précédentes):
\begin{enumerate}
  \item $A_{1}$ . $(w(3,3,4),(s_{2}s_{3})^{2}=(s_{1}s_{2}^{s_{3}})^{3}).$
  \item $A_{2}$ . $(w(3,3,6),(s_{2}s_{3})^{3}=(s_{1}s_{2}^{s_{3}})^{2}).$
  \item $A_{3}$ . $(w(3,4,6),(s_{1}s_{3})^{2}=(s_{2}s_{3})^{3}).$
  \item $A_{4}$ . $(w(6,6,4),(s_{1}s_{2})^{3}=(s_{2}s_{3})^{2}=(s_{1}s_{3})^{3}).$
  \item $A_{5}$ . $(w(6,6,6),(s_{1}s_{2}^{s_{3}})^{2}=(s_{1}s_{2})^{3}=(s_{2}s_{3})^{3}=(s_{1}s_{3})^{3}).$
\end{enumerate}
\end{proposition}
\begin{proof}
Elles ne présentent pas de difficultés, car elles sont toutes semblables à celle pour $A_{1}$.
\end{proof}
\subsection{structure de $N(G)$ et la suite $(\star\star)$}
\begin{theorem}
Le groupe $N(G)$ est un $\mathbb{Z}[\sqrt{-2}]$-module libre de rang $2$ et un $\mathbb{Z}$-module libre de rang $4$.
\end{theorem}
\begin{proof}
On obtient la structure de $N(G)$ en considérant le groupe $A_{1}$. Comme $\alpha=\beta=1$ et $\gamma=2$ dans ce cas, nous voyons que $(s_{1}z_{1})^{2}=\frac{-2}{4-\gamma}c_{1}=-c_{1}$, donc $c_{1}\in N(A_{1})$. Puis $s_{2}.c_{1}-c_{1}=c_{2}$ et $s_{3}.c_{1}-c_{1}=c_{3}$, donc $c_{1},c_{2},c_{3}$ sont dans $N(A_{1})$, d'où tous les résultats car $N(G)$ est stable par multiplication par $\alpha l=\sqrt{-2}$.
\end{proof}
\begin{theorem}
La suite $(\star\star)$ est non scindée.
\end{theorem}
\begin{proof}
Nous allons considérer le groupe $A_{1}$ ($\simeq G_{12}$). une présentation de $A_{1}$ est alors :$(w(3,3,4),(s_{1}(s_{2}s_{3})^{2})^{2})$ comme on le vérifie facilement. nous avons $\alpha=\beta=1$ , $\gamma=2$, $l=\sqrt{-2}$, $m=-\sqrt{-2}$ et la relation $(\mathcal{E})$ devient:
\[
2\lambda_{1}+\sqrt{-2}(1-\sqrt{-2})\lambda_{2}-\sqrt{-2}(1+\sqrt{-2})\lambda_{3}=-1.
\]
Celle-ci ne peut pas être satisfaite quels que soient les $\lambda_{i}$ car $\sqrt{-2}$ divise le membre de gauche de $(\mathcal{E})$ et n'est pas inversible: il ne peut pas diviser le membre de droite.
\end{proof}
\begin{remark}
Soit $W(p,q,r)$ n groupe de Coxeter de rang $3$ avec $p\geqslant3$, $q\geqslant3$ et $r\geqslant3$. On suppose que $R:W\to GL_{3}(M)$ est une représentation de réflexion affine et que $\{\alpha,\beta,\gamma\}\subset \mathbb{N}$. Alors on obtient pour $G'$ soit un groupe diédral, soit un groupe de réflexion complexe imprimitif, soit $G_{12}$.
\end{remark}
\section{Le groupe de réflexion complexe $G_{13}$ et groupes associés}
Dans tout cette section si $G$ est un groupe de réflexion, on suppose toujours que $\Delta(G)=0$.
\subsection{Généralités}
On sait que le groupe de réflexion complexe $G_{13}$ est isomorphe à une extension centrale non scindée de $W(B_{3})$ par un groupe cyclique d'ordre $2$. On voit aussi que $G_{13}$ est isomorphe au produit central $C_{4}\star \hat{S}_{4}$ avec élément d'ordre $2$ amalgamé et où $\hat{S}_{4}$ désigne une extension centrale non scindée du groupe symétrique $S_{4}$ dans laquelle les réflexions se remontent en éléments d'ordre $4$. On sait aussi que le corps de définition de $G_{13}$ comme groupe de réflexion complexe est $K=\mathbb{Q}(\zeta_{8})$ où $\zeta_{8}$ est une racine primitive $8$-ième de l'unité. Il en résulte que le corps $K_{0}=\mathbb{Q}(\sqrt{2})$ et que $K=\mathbb{Q}(\sqrt{2},i)$. On a alors $\{p,q,r\}\subset\{3,4,6,8\}$ ($2$ ne pouvant pas intervenir car $K$ n'est pas un sous-corps de $\mathbb{R}$) avec l'un de $p$, $q$ ou $r$ égal à $8$ d'après la remarque 2. Nous sommes ainsi amenés à examiner les triples suivants pour $\{p,q,r\}$:
\[
(3,3,8),(3,4,8),(3,6,8),(3,8,8),(4,4,8),(4,6,8),(4,8,8), 
\]
\[
(6,6,8),(6,8,8),(8,8,8).
\]
Nous pouvons déjà remarquer que le triple  $(4,4,8)$ n'est pas à considérer ici car il correspond à l'un des groupes étudiés au paragraphe 1; de même le triple $(4,8,8)$ qui correspond à un groupe diédral affine.
\subsection{Différents corps et groupes}
Nous cherchons d'abord tous les triples $(p,q,r)$ pour lesquels $K=\mathbb{Q}(\sqrt{2},i)$. Pour cela nous donnons les corps obtenus en prenant chacun des triples donnés plus haut. Nous rappelons que les racines de $v_{8}(X)=X^{2}-4X+2$ sont $2+\sqrt{2}$ et $2-\sqrt{2}$. Dans ce qui suit $\epsilon\in \{-1,+1\}$ ou $\epsilon\in\{+,-\}$ s'il apparait en exposant.
\begin{proposition}
On garde les hypothèses et notations précédentes.
\begin{enumerate}
  \item On obtient le corps $K_{\epsilon}=\mathbb{Q}(\sqrt{\epsilon\sqrt{2}})$ pour les triples suivants: (on appelle $A_{i}^{\epsilon}$ les groupes correspondants).
  \begin{itemize}
  \item $(3,3,8)$, $\mathcal{P}(A_{1}^{\epsilon})=\mathcal{P}(1,1,2-\epsilon\sqrt{2};\epsilon\sqrt{2}+\sqrt{\epsilon\sqrt{2}},\epsilon\sqrt{2}-\sqrt{\epsilon\sqrt{2}})$.
  \item $(3,6,8)$, $\mathcal{P}(A_{2}^{\epsilon})=\mathcal{P}(1,3,2+\epsilon\sqrt{2};-2-\epsilon\sqrt{2}+\sqrt{\epsilon\sqrt{2}},-2-\epsilon\sqrt{2}-\sqrt{\epsilon\sqrt{2}})$.
  \item $(6,6,8)$, $\mathcal{P}(A_{3}^{\epsilon})=\mathcal{P}(3,3,2-\epsilon\sqrt{2};-4+\epsilon\sqrt{2}+\sqrt{\epsilon\sqrt{2}},-4+\epsilon\sqrt{2}-\sqrt{\epsilon\sqrt{2}})$.
\end{itemize}
  \item Les groupes $A_{i}^{\epsilon}$ $((1\leqslant i \leqslant 3),\epsilon\in \{-1,+1\})$ sont isomorphes entre eux. De plus si $A$ est l'un de ces groupes, alors $A'$ est infini.
  \item L'extension $(\star\star)$ est non scindée.
\end{enumerate}
\end{proposition}
\begin{proof}
1) La démonstration ne présente pas de difficultés.

2) Partant de $A_{1}^{\epsilon}$, on a $C(s_{1},s_{2}(s_{2}s_{3})^{4})=4-\alpha=3=4-\beta=C(s_{1},s_{3}(s_{2}s_{3})^{4})$ donc, comme $<s_{2}(s_{2}s_{3})^{4},s_{3}(s_{2}s_{3})^{4}>=<s_{2},s_{3}>$, on voit que les groupes $A_{1}^{\epsilon}$ et $A_{3}^{\epsilon}$ sont isomorphes. Partant de $A_{1}^{\epsilon}$, on a $C(s_{2},s_{1}(s_{1}s_{3})^{3})=4-\alpha=3$ et  $C(s_{2},s_{3}(s_{1}s_{3})^{4})=4-\gamma=2-\epsilon\sqrt{2}$ donc, comme $<s_{1}(s_{1}s_{3})^{3},s_{3}(s_{1}s_{3})^{3}>=<s_{1},s_{3}>$, on voit que les groupes $A_{2}^{\epsilon}$ et $A_{3}^{\epsilon}$ sont isomorphes.\\
Nous montrons maintenant que les groupes $A_{1}^{+}$ et $A_{1}^{-}$ sont isomorphes.\\
Les quatre racines du polynôme $P(X)=X^{4}-2$ sont $\sqrt[4]{2}$, $i\sqrt[4]{2}$, $-\sqrt[4]{2}$, $-i\sqrt[4]{2}$. Soit $\tilde{K}$ un corps de décomposition de $P(X)$. Le groupe de Galois  $\mathcal{G}(\tilde{K}/\mathbb{Q})$ est isomorphe au groupe diédral d'ordre $8$. Considérons $\tilde{M}:=M\otimes_{K}\tilde{K}$. Soit $\mathcal{A}:=(a_{1},a_{2},a_{3})$ une base du $\tilde{K}$-espace vectoriel $\tilde{M}$ sur lequel les éléments $s_{1}$, $s_{2}$, $s_{3}$ de $A_{1}^{-}$ opèrent:
\[
\begin{array}{c|ccc}
& s_{1}.a_{1} & = & -a_{1}\\
& s_{1}.a_{2} & = & a_{2}+a_{1}\\
& s_{1}.a_{3} & = & a_{3}+a_{1}
\end{array}
,
\begin{array}{c|ccc}
& s_{2}.a_{1} & = & a_{1}+ a_{2}\\
& s_{2}.a_{2} & = & -a_{2}\\
& s_{2}.a_{3} & = & a_{3}+l_{1}a_{2}
\end{array}
,
\begin{array}{c|ccc}
& s_{3}.a_{1} & = & a_{1}+a_{3}\\
& s_{3}.a_{2} & = & a_{2}+m_{1}a_{3}\\
& s_{3}.a_{3} & = & -a_{3}
\end{array}
\]
avec $l_{1}=\sqrt{2}+\sqrt[4]{2}$ et $m_{1}=\sqrt{2}-\sqrt[4]{2}$.\\
Soit $\mathcal{B}:=(b_{1},b_{2},b_{3})$ une base du $\tilde{K}$-espace vectoriel $\tilde{M}$ sur lequel les éléments $t_{1}$, $t_{2}$, $t_{3}$ opèrent:
\[
\begin{array}{c|ccc}
& t_{1}.b_{1} & = & -b_{1}\\
& t_{1}.b_{2} & = & b_{2}+b_{1}\\
& t_{1}.b_{3} & = & b_{3}+b_{1}
\end{array}
,
\begin{array}{c|ccc}
& t_{2}.b_{1} & = & b_{1}+b_{2}\\
& t_{2}.b_{2} & = & -b_{2}\\
& t_{2}.b_{3} & = & b_{3}+l_{2}b_{2}
\end{array}
,
\begin{array}{c|ccc}
& t_{3}.b_{1} & = & b_{1}+b_{3}\\
& t_{3}.b_{2} & = & b_{2}+m_{2}b_{3}\\
& t_{3}.b_{3} & = & -b_{3}
\end{array}
\]
avec $l_{2}=\sqrt{2}+i\sqrt[4]{2}$ et $m_{2}=\sqrt{2}-i\sqrt[4]{2}$.

Soit $\theta\in\mathcal{G}(\tilde{K}/\mathbb{Q})$ défini par $\theta(\sqrt[4]{2})=i\sqrt[4]{2}$ et $\theta(i)=i$. On a $\theta(\sqrt{2})=-\sqrt{2}$, $\theta(l_{1})=l_{2}$ et $\theta(m_{1})=m_{2}$. Soit $\sigma:\tilde{M}\to\tilde{M}$ l'application semi-linéaire définie par $\sigma(a_{i})=b_{i}$ $(1\leqslant i \leqslant 3)$. Si $\lambda\in \tilde{K}$ et $\nu\in\tilde{M}$ on a $:\sigma(\lambda\nu)=\theta(\lambda)\sigma(\nu)$. Il est alors facile de voir que pour $1\leqslant i \leqslant 3$, on a $\sigma s_{i}\sigma^{-1}=t_{i}$ donc $\sigma A_{1}^{-}\sigma^{-1}=A_{1}^{+}$: les groupes $A_{1}^{-}$ et  $A_{1}^{+}$ sont isomorphes.

3) Considérons $A_{1}^{-}$. On a $C(s_{1},s_{2}^{s_{3}})=3+\sqrt{2}$, donc $s_{1}s_{2}^{s_{3}}$ est d'ordre infini et n'appartient pas à $N(A_{1}^{-})$ donc son image est aussi d'ordre infini dans $(A_{1}^{-})'$.

4) Dans $A_{1}^{+}$ la relation $(\mathcal{E})$ devient: $-1=(2-\sqrt{2})\lambda_{1}+(2+\sqrt{2}+\sqrt[4]{2})\lambda_{2}+(2+\sqrt{2}-\sqrt[4]{2})\lambda_{3}$. Cette relation ne peut pas être satisfaite car $\sqrt[4]{2}$ divise le membre de droite de $(\mathcal{E})$ et n'est pas inversible.
\end{proof}
\begin{proposition}
On garde les hypothèses et notations précédentes. On obtient le corps $K=\mathbb{Q}(\sqrt{1-\epsilon\sqrt{2}})$ pour les triples suivants: (on appelle $B_{i}^{\epsilon}$ les groupes correspondants):
\begin{itemize}
  \item $(8,8,8)$, $\mathcal{P}(B_{1}^{\epsilon})=\mathcal{P}(2+\epsilon\sqrt{2},2+\epsilon\sqrt{2},2+\epsilon\sqrt{2};-2-3\epsilon\sqrt{2}+\sqrt{2}\sqrt{1-\epsilon\sqrt{2}},-2-3\epsilon\sqrt{2}-\sqrt{2}\sqrt{1-\epsilon\sqrt{2}})$
  \item $(8,8,8)$, $\mathcal{P}(B_{1}^{\epsilon})=\mathcal{P}(2-\epsilon\sqrt{2},2-\epsilon\sqrt{2},2+\epsilon\sqrt{2};-2+\epsilon\sqrt{2}+\sqrt{2}\sqrt{1-\epsilon\sqrt{2}},-2+\epsilon\sqrt{2}-\sqrt{2}\sqrt{1-\epsilon\sqrt{2}})$
\end{itemize}
Les groupes obtenus sont tous isomorphes entre eux. Si $B$ est l'un d'eux, $B'$ est infini. La suite $(\star\star)$ est non scindée.
\end{proposition}
\begin{proof}
Elle est semblable à celle de la proposition 10. le corps $\tilde{K}$ est un corps de décomposition du polynôme $X^{4}-2X^{2}-1$. Nous montrons maintenant que la suite $(\star\star)$ est non scindée. Nous nous pla\c cons dans le cas $K=\mathcal{Q}(\sqrt{2+2\sqrt{2}})$, $\alpha=\beta=2+\sqrt{2}$, $\gamma=2-\sqrt{2}$ et $\theta=\sqrt{2+2\sqrt{2}}$.\\
Si $\zeta\in  N$ alors $\gamma\zeta \in N$ et, comme $2\zeta\in N$, nous voyons que $\sqrt{2}\zeta \in N$; de plus $\theta\zeta\in N$. On a déjà vu que $(s_{1}(s_{2}s_{3})^{4})^{2}=(\frac{-2}{4-\gamma})c_{1}\in [N,s_{1}]$ et comme $2=\gamma(4-\gamma)$, nous voyons que $\gamma c_{1}\in [N,s_{1}]$ et donc $\gamma(4-\gamma)c_{1}=2c_{1}\in [N,s_{1}]$. Il en résulte que $\sqrt{2}c_{1}\in[N,s_{1}]$ et $\sqrt{2}\theta c_{1}=2\sqrt{1+\sqrt{2}}c_{1}\in[N,s_{1}]$.\\
Nous avons $y_{4}^{2}=(s_{3}(s_{1}s_{2})^{4})^{2}=(\frac{-2}{m+2})c_{3}\in [N,s_{3}]$ et comme $(\frac{-2}{m+2})=\frac{\gamma(4-\gamma)}{l}=\beta m$ nous obtenons $y_{4}^{2}=\beta mc_{3}\in[N,s_{3}]$, d'où aussi $2mc_{3}\in[N,s_{3}]$. \\
De même $x_{4}^{2}=(s_{2}(s_{1}s_{3})^{4})^{2}=(\frac{-2}{l+2})c_{2}=\alpha lc_{2}\in[N,s_{2}]$ d'où aussi $2lc_{2}\,[N,s_{2}]$. On obtient alors $2lc_{1}\in [N,s_{1}]$ et $2mc_{1}\in [N,s_{1}]$. Dans tous les cas, nous voyons que si $\lambda c_{i}\in [N,s_{1}]$, alors $\sqrt{2}$ divise $\lambda$. Comme dans la proposition précédente ceci est impossible.
\end{proof}
On peut remarquer que les corps obtenus dans les  cas précédents contiennent le sous-corps $\mathbb{Q}(\sqrt{2},i)$.
\subsection{Le corps $\mathbb{Q}(\sqrt{2},\sqrt{3})$}
Les triples $(8,8,3)$ et $(8,8,6)$.
\subsubsection{Quelques généralités. Résultats indépendants du corps $K$}
\begin{remark}
Contrairement à la convention précédente concernant $r$ nous supposerons que $\alpha$ et $\beta$ sont des racines de $v_{8}(X)$ et $r\in \{3,6\}$.\\ 
On utilise les notations et résultats de la notation 2 2) de \cite{Z5}. Soient $z_{2}=(s_{1}s_{3})^{4}$, $s_{12}=s_{1}z_{2}$ et $s_{32}=s_{3}z_{2}$. Comme $<s_{2},s_{3}>=<s_{12},s_{32}>=<s_{1},s_{32}>=<s_{12},s_{3}>$, on obtient $G=(G_{2}')=<s_{12},s_{2},s_{32}>=(\Gamma_{2})=<s_{1},s_{2},s_{32}>=(\Gamma_{2}')=<s_{12},s_{2},s_{3}>$ avec les systèmes de paramètres $\mathcal{P}(G_{2}')=\mathcal{P}(\alpha',\beta,\gamma';-\beta(m+2),-(\alpha l+2\beta))$, $\mathcal{P}(\Gamma_{2})=\mathcal{P}(\alpha,\beta',\gamma';-\alpha(l+2),-(2\alpha+\beta m))$, $\mathcal{P}(\Gamma_{2}')=\mathcal{P}(\alpha',\beta',\gamma ;-l(\alpha+2m),-m(\beta+2l))$. On rappelle que $\alpha'=4-\alpha$, $\beta'=4-\beta$ et $\gamma'=4-\gamma$ et aussi que $\alpha$ et $\alpha'$ (resp. $\beta$ et $\beta'$) sont les racines de $v_{8}(X)=X^{2}-4X+2$. De plus si $\gamma=1$ (racine de $v_{3}(X)=X-1$) alors $\gamma'=3$ (racine de $v_{6}(X)=X-3$). La forme $\Gamma_{2}'$ montre que l'on peut supposer que $\alpha=2+\sqrt{2}$.\\
\begin{itemize}
  \item Si $r=3$ alors $\gamma=1$ et pour $G_{2}'$ on a $\gamma'=3$.
  \item Si $r=6$ alors $\gamma=3$ et pour $G_{2}'$ on a $\gamma'=1$.
\end{itemize}
\end{remark}

Ici un nouveau phénomène se produit. En effet si $R:W(8,8,3)\to GL(M)$ est une représentation de réflexion affine et si $\alpha$ et $\beta$ sont des racines de $v_{8}(X)$, on obtient des corps différents pour $K$ suivant que $\alpha=\beta$ ou non.
\begin{proposition}
Soient le groupe $W(8,8,3)$ et $R(\alpha,\beta, \gamma ;l)$ une représentation de réflexion de ce groupe.
\begin{enumerate}
  \item On a les équivalences:
  \begin{itemize}
  \item (A) \quad $\Delta=0$;
  \item (B) \quad $\alpha l+\beta m=6-2(\alpha+\beta)$;
  \item (C) \quad $s_{2}s_{3}^{s_{1}s_{3}s_{1}}$ est d'ordre $6$;
  \item (D) \quad $s_{3}s_{2}^{s_{1}s_{2}s_{1}}$ est d'ordre $6$.
\end{itemize}
  \item Si ces conditions sont satisfaites, on a:
  \begin{itemize}
  \item (i) \quad $C(s_{2},s_{3}^{s_{1}s_{3}s_{1}})=4-\alpha$ et $s_{2}s_{1}^{s_{3}s_{3}s_{3}}$ est d'ordre $8$;
  \item (ii) \quad $C(s_{3},s_{1}^{s_{2}s_{1}s_{2}})=4-\beta$ et $s_{3}s_{1}^{s_{2}s_{1}s_{2}}$ est d'ordre $8$.
\end{itemize}
De plus $\alpha l$ et $\beta m$ sont les racines du polynôme:
\[
Q(X)=(X-3+\alpha+\beta)^{2}-\alpha\beta+2(\alpha+\beta)-5.
\]
\begin{itemize}
  \item Si $\alpha\neq\beta$, $\alpha\beta=2$ et $\alpha+\beta=4$: $\alpha l+\beta m=-2$, $Q(X)=(X+1)^{2}+1$, $\alpha l=-1+i$, $\beta m=-1-i$ et $K=\mathbb{Q}(\sqrt{2},i)$.
  \item Si $\alpha=\beta$, $-\alpha\beta+2(\alpha+\beta)-5=-3$: $\alpha l=\beta m=6-4\alpha$, $Q(X)=(X-3+2\alpha)^{2}-3$, $\alpha l=3-2\alpha+\sqrt{3}$, $\beta m=3-2\alpha-\sqrt{3}$ et $K=\mathbb{Q}(\sqrt{2},\sqrt{3})$.
 \end{itemize}
\end{enumerate}
\end{proposition}
\begin{proof}
On a $\Delta=8-2(\alpha+\beta)-2-(\alpha l+\beta m)$ donc (A) et (B) sont équivalentes.\\
On a $C(s_{2},s_{3}^{s_{1}s_{3}s_{1}})=2\alpha\beta+(\beta-1)^{2}+\beta(\alpha l+\beta m)$, mais $(\beta-1)^{2}=2\beta-1$ donc $C(s_{2},s_{3}^{s_{1}s_{3}s_{1}})=2\alpha\beta+2\beta-1+\beta(\alpha l+\beta m)$.
 Si $\alpha l+\beta m=6-2(\alpha+\beta)$ alors $C(s_{2},s_{3}^{s_{1}s_{3}s_{1}})=2\alpha\beta+2\beta-1+6\beta-2\alpha\beta-2\beta^{2}=-2\beta^{2}+8\beta-1=3$ et $s_{2}s_{3}^{s_{1}s_{3}s_{1}}$ est d'ordre $6$. On voit de même que $s_{3}s_{2}^{s_{1}s_{2}s_{1}}$ est d'ordre $6$. Si $s_{2}s_{3}^{s_{1}s_{3}s_{1}}$ est d'ordre $6$ alors $C(s_{2},s_{3}^{s_{1}s_{3}s_{1}})=3$ donc $2\alpha\beta+2\beta-1+\beta(\alpha l+\beta m)=3$ et $2\alpha\beta+2\beta-+\beta(\alpha l+\beta m)=4=2\beta(4-\beta)$ car les deux racines de $v_{8}(X)$ sont $\beta$ et $4-\beta$. Nous obtenons $2\alpha+2+(\alpha l+\beta m)=2(4-\beta)$, ou encore $\alpha l+ \beta m=6-2(\alpha+\beta)$: les conditions (B) et (C) sont donc équivalentes. De même les conditions (B) et (D) sont équivalentes. Si ces conditions sont satisfaites alors $C(s_{2},s_{3}^{s_{3}s_{1}s_{3}})= 2\beta+\alpha(\beta-1)^{2}+\beta(\alpha l+\beta m)=-2\beta^{2}+8\beta-\alpha=4-\alpha$ et $s_{2}s_{3}^{s_{3}s_{1}s_{3}}$ est d'ordre $8$. Par symétrie nous avons le (ii).\\
 Nous avons $\alpha l\beta m=\alpha\beta$, donc $\alpha l$ et $\beta m$ sont les racines du polynôme 
 \[
 Q(X)=X^{2}-2(3-(\alpha+\beta))X+\alpha\beta=(X-3+\alpha+\beta)^{2}-\alpha\beta+2(\alpha+\beta)-5.
 \]
 \begin{itemize}
  \item Si $\alpha\neq\beta$, alors $\alpha+\beta=4$, $\alpha\beta=2$, $\alpha l+\beta m=-2$ et $Q(X)=(X+1)^{2}+1$ d'où $K$ et $\alpha l$ et $\beta m$.
  \item Si $\alpha=\beta$, alors $-\alpha\beta+2(\alpha+\beta)-5=-3$, $\alpha l+\beta m=6-4\alpha$, $Q(X)=(X-3+2\alpha)^{2}-3$, d'où $K$, $\alpha l$ et $\beta m$.
\end{itemize}
 \end{proof}
\begin{proposition}
On obtient le corps $K=\mathbb{Q}(\sqrt{2},\sqrt{3})$ pour les triples suivants avec une condition supplémentaire (on appelle $C_{i}^{\epsilon}$ les groupes correspondants):
\begin{itemize}
  \item $(8,8,3)$, $\mathcal{P}(C_{1}^{\epsilon})=\mathcal{P}(2+\epsilon\sqrt{2},2+\epsilon\sqrt{2},1;-1-2\epsilon\sqrt{2}+\sqrt{3},-1-2\epsilon\sqrt{2}-\sqrt{3}).$
  \item $(8,8,6)$, $\mathcal{P}(C_{2}^{\epsilon})=\mathcal{P}(2+\epsilon\sqrt{2},2-\epsilon\sqrt{2},3;-3+\sqrt{3},-3-\sqrt{3}).$
\end{itemize}
Les groupes obtenus sont isomorphes entre eux. Si $C$ est l'un de ces groupes alors $C'$ est infini.
\end{proposition}
\begin{proof}
Elle est semblable à celle des propositions 11 et 12.
\end{proof}
\begin{corollary}
On a les ``présentations'' de $C$ suivantes: 
\[
(w(8,8,3),(s_{2}s_{3}^{s_{1}s_{3}s_{1}})^{6}=1\; \text{et}\; \beta=\alpha);
\]
\[
(w(8,8,6),(s_{2}s_{3}^{s_{1}s_{3}s_{1}})^{3}=1\; \text{et}\; \beta=\alpha).
\]
\end{corollary}
\begin{proof}
C'est clair d'après ce qui précède.
\end{proof}
On peut remarquer que comme $\Delta(C)=0$ et $\alpha l\neq\beta m$ avec un corps $K$ réel, il n'y a aucune forme bilinéaire non nulle $C$-invariante.
\subsubsection{Structure de $N$ et la suite $(\star\star)$.}
L'anneau des entiers de $\mathbb{Q}(\sqrt{2},\sqrt{3})$ est $\mathcal{O}:=\mathcal{O}(\mathbb{Q}(\sqrt{2},\sqrt{3})=\{\frac{1}{2}(2a+b\sqrt{2}+2c\sqrt{3}+d\sqrt{6}|(a,b,c,d)\in \mathbb{Z}^{4}\}$. C'est un anneau principal.
 \begin{proposition}
On a: $I_{1}=I_{2}=I_{3}:=I$ est l'idéal de $\mathcal{O}$ engendré par $\sqrt{2}$ . $N$ est un $\mathcal{O}G$-module libre de rang $2$ et un $\mathbb{Z}$-module libre de rang $8$.
\end{proposition}
\begin{proof}
Comme $l=\frac{1}{2}(-6+3\sqrt{2}+2\sqrt{3}-\sqrt{6})$ et $m=\frac{1}{2}(-6-3\sqrt{2}-2\sqrt{3}-\sqrt{6})$, nous voyons que $l$ et $m$ sont des entiers algébriques. De même $\frac{2}{l+2}=-2+3\sqrt{2}-2\sqrt{3}+\sqrt{6}$ et $\frac{2}{m+2}=-2-3\sqrt{2}+2\sqrt{3}+\sqrt{6}$ sont des entiers algébriques.
Si $\zeta\in N$, alors $\alpha\zeta$, $\beta\zeta$, $\gamma\zeta$ et $\theta\zeta$ sont dans $N$ donc si $\zeta\in N$ alors $\sqrt{2}\zeta$, $\sqrt{3}\zeta$ et  $\sqrt{6}\zeta$ sont dans $N$: $N$ est un $\mathcal{O}G$-module. Les ensembles $I_{j}$ sont des idéaux de $\mathcal{O}$, ce sont donc des idéaux principaux.

Nous avons:
\begin{itemize}
  \item $(s_{1}(s_{2}s_{3})^{3})^{2}=\frac{2}{4-\gamma}c_{1}=2c_{1}$, donc $2\in I_{1}$;
  \item $(s_{2}(s_{3}s_{1})^{4})^{2}=\frac{-2}{l+2}c_{2}$, donc $\frac{2}{l+2}\in I_{2}\subset I_{1}$;
  \item $(s_{3}(s_{1}s_{2})^{4})^{2}=\frac{-2}{m+2}c_{3}$, donc $\frac{2}{m+2}\in I_{3}\subset I_{1}.$
\end{itemize}
Nous venons de voir que $2\in I_{1}$ donc $2\alpha=4+2\sqrt{2}\in I_{1}$ et $2\sqrt{2}\in I_{1}$; de même $2\alpha l=-6+3\sqrt{3}\in I_{1}$ et $2\sqrt{3}\in I_{1}$.\\
Comme $\frac{2}{l+2}\in I_{1}$ on voit ainsi que $\sqrt{2}+\sqrt{6}$ et $\sqrt{2}-\sqrt{6}$ sont dans $I_{1}$  et $I_{1}$ est l'idéal principal engendré par $\sqrt{2}$.

Nous avons déjà vu que $\frac{2}{l+2}$ est dans $I_{2}$ donc:
\begin{gather*}
(1) = 2-3\sqrt{2}+2\sqrt{3}-\sqrt{6}\in I_{2} \\
(2)  =  \sqrt{2}(1) = -6+2\sqrt{2}-2\sqrt{3}+2\sqrt{6}\in I_{2} \\
(3)  =  \sqrt{3}(1) = 6-3\sqrt{2}+2\sqrt{3}-3\sqrt{6}\in I_{2} \\
(4)  =  \sqrt{6}(1)= -6+6\sqrt{2}-6\sqrt{3}+2\sqrt{6}\in I_{2} 
\end{gather*}
Nous étudions maintenant ces quatre équations: $(2)+(3)=-\sqrt{2}-\sqrt{6}\in I_{2}$, donc $\sqrt{2}+\sqrt{6}\in I_{2}$, $\sqrt{3}(\sqrt{2}+\sqrt{6})=3\sqrt{2}+\sqrt{6}\in I_{2}$ et $2\sqrt{2}\in I_{2}$ et $\sqrt{6}-\sqrt{2}\in I_{2}$: $2\sqrt{6}\in I_{2}$; de $(1)$ on tire alors $2+2\sqrt{3}\in I_{2}$.\\
Nous avons des résultats analogues pour $I_{3}$.\\
Nous utilisons maintenant les inclusions $lI_{2}\subset I_{3}$ et $mI_{3}\subset I_{2}$. Nous avons $l(\sqrt{6}+\sqrt{2})=2\sqrt{+}-2\sqrt{6}\in I_{3}$. Comme $2\sqrt{6}\in I_{3}$ nous voyons que $2\sqrt{3}\in I_{2}$ d'où $2\in I_{3}$: $I_{3}$ (resp. $I_{2}$) est l'idéal principal engendré par $\sqrt{2}$. Tout le reste est clair.
\end{proof}
\begin{corollary}
La suite $(\star\star)$ est non scindée.
\end{corollary}
\begin{proof}
Le membre de droite de la suite $(\star\star)$ est dans $I$ donc on a une impossibilité.
\end{proof}
\subsection{Le corps $\mathbb{Q}(\sqrt{2},i)$}
Les triples $(8,8,3)$, $(8,8,6)$, $(3,4,8)$ et $(4,6,8)$.
\subsubsection{Généralités. Le groupe $G$}
\begin{proposition}
On obtient le corps $\mathbb{Q}(\sqrt{2},i)$ pour les triples suivants (certains avec une condition supplémentaire):
\begin{itemize}
  \item $D_{1}^{\epsilon}$: $(p,q,r)=(8,8,3)$, $\mathcal{P}(D_{1}^{\epsilon})=\mathcal{P}(2+\epsilon\sqrt{2},2-\epsilon\sqrt{2},1;-1+i,-1-i)$;
  \item $D_{2}^{\epsilon}$: $(p,q,r)=(8,8,6)$, $\mathcal{P}(D_{2}^{\epsilon})=\mathcal{P}(2+\epsilon\sqrt{2},2+\epsilon\sqrt{2},3;-3-2\epsilon\sqrt{2}+i,-3-2\epsilon\sqrt{2}-i)$;
  \item $D_{3}^{\epsilon}$: $(p,q,r)=(3,4,8)$, $\mathcal{P}(D_{3}^{\epsilon})=\mathcal{P}(1,2,2+\epsilon\sqrt{2};-1-\epsilon\sqrt{2}+i,-1-\epsilon\sqrt{2}-i)$;
  \item $D_{4}^{\epsilon}$: $(p,q,r)=(4,6,8)$, $\mathcal{P}(D_{4}^{\epsilon})=\mathcal{P}(2,3,2+\epsilon\sqrt{2};-3-\epsilon\sqrt{2}+i,-3-\epsilon\sqrt{2}-i)$.
\end{itemize}
Dans les quatre cas, si $t=s_{1}s_{2}s_{3}$ et si $t'$ est l'image de $t$ dans $G'$, on a $t'^{3}=i\,id_{M'}$.
\end{proposition}
\begin{proof}
Elle ne présente pas de difficultés.
\end{proof}
\begin{proposition}
Soient le groupe $W(8,8,3)$ et $R(\alpha,\beta,\gamma ;l)$ une représentation de réflexion de ce groupe. Alors $\alpha,\beta$ sont racines de $v_{8}(X)$ et $\gamma=1$. De plus deux des conditions suivantes impliquent les deux autres:
 \begin{itemize}
  \item (A) \;$\Delta=0$;
  \item (B) \;$\beta \neq \alpha$;
  \item (C) \;$(s_{2}s_{3}^{s_{1}})^{3}=1$;
  \item (D) \;$(s_{2}s_{1}^{s_{3}})^{4}=1$. 
  \end{itemize}
\end{proposition}
\begin{proof}
1) Supposons les conditions (A) et (B) satisfaites. Alors $C(s_{2},s_{3}^{s_{1}}) =\alpha\beta+\gamma+(\alpha l+\beta m)=2+1-2=1$ et $s_{2}s_{3}^{s_{1}}$ est d'ordre $3$; $C(s_{2},s_{1}^{s_{3}}) =\beta\gamma+\alpha+(\alpha l+\beta m)=4-2=2$ et $s_{2}s_{1}^{s_{3}}$ est d'ordre $4$.\\
2) Supposons les conditions (A) et (C) satisfaites. Comme on ne peut pas avoir $s_{2}=s_{3}^{s_{1}}$, $(s_{2}s_{3}^{s_{1}})^{3}$ est d'ordre $3$ et $C(s_{2},s_{3}^{s_{1}})=\alpha\beta+1+6-2(\alpha+\beta)=1$, d'où $\alpha\neq\beta$. Dans ces conditions  on a aussi $s_{2}s_{1}^{s_{3}}$  d'ordre $4$ d'après la condition (A).\\
3) Supposons les conditions (A) et (D) satisfaites. Alors $C(s_{2},s_{1}^{s_{3}}) =\beta+\alpha+(\alpha l+\beta m)=(l+1)(\alpha+\beta m)$ donc si $s_{2}s_{1}^{s_{3}}$ est d'ordre $2$, on a $l+1=\alpha+\beta m=0$ donc $l=-1=m$ et $\alpha=\beta$. C'est impossible, car alors $\Delta\neq 0$; donc $s_{2},s_{1}^{s_{3}}$ est d'ordre $4$ et $\beta+\alpha+(\alpha l+\beta m)=2=6-(\alpha+\beta)$: $\alpha+\beta=4$ et $\beta \neq\alpha$. Dans ces conditions, on a aussi $(s_{2}s_{3}^{s_{1}})^{3}=1$.\\
4) Supposons les conditions (B) et (C) satisfaites. Alors $C(s_{2},s_{3}^{s_{1}}) =\alpha\beta+1+(\alpha l+\beta m)=1$ donc $\alpha l+\beta m=-2$ et $\Delta=0$ On a donc aussi $(s_{2}s_{1}^{s_{3}})$ est d'ordre $4$.\\
5) Supposons les conditions (B) et (D) satisfaites. Comme ci-dessus, c'est impossible si $s_{2}s_{3}^{s_{1}}$ est d'ordre $2$. Alors $C(s_{2},s_{1}^{s_{3}}) =2=\alpha +\beta+(\alpha l+\beta m)$ donc $\alpha l+\beta m=-2$ et nous concluons comme ci-dessus.\\
6) Supposons les conditions (C) et (D) satisfaites. Si $s_{2}s_{1}^{s_{3}}$  est d'ordre $2$, alors $C(s_{2},s_{3}^{s_{1}}) =(\alpha+m)(\beta+l)=(\alpha-1)^{2}=2\alpha-1\neq 1$ et ceci est impossible.Donc nous avons $C(s_{2},s_{3}^{s_{1}}) =\alpha\beta+1+(\alpha l+\beta m)=1$ et $C(s_{2},s_{1}^{s_{3}}) =\beta+\alpha+(\alpha l+\beta m)=2$. De la première équation nous tirons $\alpha l+\beta m=-\alpha\beta$, d'où, en reportant dans la deuxième équation: $2=\alpha+\beta-\alpha\beta$ et la seule possibilité est $\beta\neq\alpha$. Nous pouvons appliquer le (B) pour conclure.
\end{proof}
\begin{proposition}
Soient le groupe $W(8,8,6)$ et $R(\alpha,\beta,\gamma ;l)$ une représentation de réflexion de ce groupe. Alors $\alpha,\beta$ sont racines de $v_{8}(X)$ et $\gamma=3$. De plus deux des conditions suivantes impliquent les deux autres:
 \begin{itemize}
  \item (A) \; $\Delta=0$;
  \item (B) \; $\beta = \alpha$;
  \item (C) \;$(s_{1}s_{2}^{s_{3}})^{4}=1$;
  \item (D) \;$(s_{2}s_{3}^{s_{1}s_{3}})^{3}=1$. 
  \end{itemize}
\end{proposition}
\begin{proof}
1) Remarquons d'abord que $s_{1}s_{2}^{s_{3}}=1$ et $s_{2}s_{3}^{s_{1}s_{3}}=1$ ne sont pas possibles car $\mathcal{A}$ est une base de $M$. Ensuite $s_{1}s_{2}^{s_{3}}$ ne peut pas être d'ordre $2$, en effet comme $C(s_{1},s_{2}^{s_{3}})=(l+1)(\alpha+\beta m)$ on a $s_{1}s_{2}^{s_{3}}$ d'ordre $2$ si et seulement si $l+1=\alpha+\beta m=0$ donc $l=-1$, $m=-3$ car $lm=\gamma=3$ et $0=\alpha-3\beta$ ce qui n'est pas. Il en résulte que dans 3. $s_{1}s_{2}^{s_{3}}$ est d'ordre $4$.\\
2) Supposons les conditions (A) et (B) satisfaites. Alors $\Delta=8-4\alpha-6-(\alpha l + \beta m)=0$ $\iff \alpha l+\beta m=2-4\alpha$. Dans ces conditions $C(s_{1},s_{2}^{s_{3}})=4\alpha+2-4\alpha=2$: $s_{1}s_{2}^{s_{3}}$ est d'ordre $4$ et $C(s_{2},s_{3}^{s_{1}s_{3}})=\alpha^{2}+3(\alpha-1)^{2}+(\alpha-1)(2-4\alpha)=1$: $s_{2}s_{3}^{s_{1}s_{3}}$ est d'ordre $3$. Les conditions (C) et (D) sont satisfaites.\\
3) Supposons les conditions (A) et (C) satisfaites. D'après le 1) $s_{1}s_{2}^{s_{3}}$ est d'ordre $4$, donc $\alpha+3\beta+(\alpha l+\beta m)=2=\alpha +3\beta+2-4\alpha$ d'où $\beta=\alpha$: la condition (B) est satisfaite, donc aussi la condition (D).\\
4) Supposons les conditions (A) et (D) satisfaites. D'après le 1) $s_{2}s_{3}^{s_{1}s_{3}}$ est d'ordre $4$, donc $2=\alpha\beta+3(\beta-1)^{2}+(\beta-1)(\alpha l+\beta m)$. Ici $\Delta=0=8-2\alpha-2\beta-6-(\alpha l+\beta m)$ et $(\alpha l+\beta m)=2-2\alpha-2\beta$. Nous avons $C(s_{2},s_{3}^{s_{1}s_{3}})=\alpha\beta+3(2\beta-1)+(\beta-1)(2-2\alpha-2\beta)$.  Nous en déduisons, après simplifications, $2=-\alpha\beta+2\alpha+2\beta$. On ne peut pas avoir $\beta \neq \alpha$ car alors $\alpha\beta=2$, $\alpha+\beta=4$ et on aboutit à une impossibilité; donc $\beta=\alpha$ et la condition (B) est satisfaite. Il en est de même de la condition (C).\\
5) Supposons les conditions (B) et (C) satisfaites. Nous avons $C(s_{1},s_{2}^{s_{3}})=2=4\alpha+(\alpha l+\beta m)$ donc $\alpha l+\beta m=2-4\alpha$, ce qui implique que $\Delta=0$ . C'est la condition (A). Dans ces conditions, la condition (D) est aussi satisfaite.\\
6) Supposons les conditions (B) et (D) satisfaites. Nous avons $C(s_{2},s_{3}^{s_{1}s_{3}})=1=\alpha^{2}+3(\alpha-1)^{2}+(\alpha-1)(\alpha l+\beta m)$ d'où, après simplifications: $0=10\alpha-6+(\alpha-1)(\alpha l+\beta m)$. Comme $(\alpha-1)(\alpha-3)=1$, nous obtenons $0=(\alpha-3)(10\alpha-6)+(\alpha l+\beta m)=4\alpha-2+(\alpha l +\beta m)$ et la condition (A) est satisfaite, d'où aussi la condition (C).\\
7) Supposons les conditions (C) et (D) satisfaites. On a $C(s_{1},s_{2}^{s_{3}})=2=\alpha+3\beta+(\alpha l+\beta m)$ et $C(s_{2},s_{3}^{s_{1}s_{3}})=1=\alpha\beta+3(2\beta-1)+(\beta-1)(\alpha l+\beta m)$. De la première équation on tire $\alpha l+\beta m=2-\alpha-3\beta$ puis 
\begin{align*}
1 & = & \alpha\beta+3(2\beta-1)+(\beta-1)(2-\alpha-3\beta)\\
   & = & \alpha\beta+6\beta-3+2\beta-2-\alpha\beta+\alpha-3\beta^{2}+3\beta
\end{align*}
d'où $6=11\beta-3\beta^{2}+\alpha=11\beta-12\beta+6+\alpha=6+\alpha-\beta$ donc $\beta=\alpha$ et la condition (B) est satisfaite. Il en est de même de la condition (A).
\end{proof}
\begin{proposition}
Soient le groupe $W(3,4,8)$ et $R(\alpha,\beta,\gamma ;l)$ une représentation de réflexion de ce groupe. On a $\alpha=1$, $\beta=2$ et $\gamma$ racine de $v_{8}(X)$. Les deux conditions suivantes sont équivalentes:
\begin{itemize}
  \item (A) \; $\Delta=0$;
  \item (B) \; $(s_{1}s_{2}^{s_{3}s_{2}})^{3}=1$.
\end{itemize}
\end{proposition}
\begin{proof}
1) La condition (A) implique la condition (B). On a $\Delta=8-2-4-2\gamma-(l+2m)$ donc $\Delta=0\iff l+2m=2-2\gamma$. Ensuite $C(s_{1},s_{2}^{s_{3}s_{2}})=\beta\gamma+\alpha(\gamma-1)^{2}+(\gamma-1)(\alpha l+\beta m)$ donc $C(s_{1},s_{2}^{s_{3}s_{2}})=2\gamma+2\gamma-1+2\gamma-2-2\gamma^{2}+2\gamma=-2\gamma^{2}+8\gamma-3=1$ et $s_{1}s_{2}^{s_{3}s_{2}}$ est d'ordre $3$.\\
2) La condition (B) implique la condition (A). Comme dans les propositions précédentes on voit que $(s_{1}s_{2}^{s_{3}s_{2}})^{3}$ est d'ordre $3$, donc $C(s_{1},s_{2}^{s_{3}s_{2}})=1=2\gamma+2\gamma-1+(\gamma-1)(l+2m)$ et $2-4\gamma=(\gamma-1)(l+2m)$. On obtient $(2-4\gamma)(\gamma-3)=l+2m=-4\gamma^{2}+14\gamma-6=-2+2\gamma$ et d'après le 1) $\Delta=0$: la condition (A) est satisfaite.
\end{proof}
\begin{proposition}
Soient le groupe $W(4,6,8)$ et $R(\alpha,\beta,\gamma ;l)$ une représentation de réflexion de ce groupe. On a $\alpha=2$, $\beta=3$ et $\gamma$ racine de $v_{8}(X)$. Les deux conditions suivantes sont équivalentes:
\begin{itemize}
  \item (A) \; $\Delta=0$;
  \item (B) \; $(s_{1}s_{3}^{s_{2}})^{3}=1$.
\end{itemize}
\end{proposition}
\begin{proof}
Elle est semblable à la précédente.
\end{proof}
\begin{proposition}
Il existe un groupe $G$ tel que $D_{i}^{\epsilon}$ soit isomorphe à $G$ pour $1\leqslant i\leqslant 4$ et $\epsilon\in \{+,-\}$.
\end{proposition}
\begin{proof}
On a déjà vu le résultat lorsque $i\in\{1,2\}$ et $\epsilon\in \{+,-\}$.\\
Partant de $D_{1}^{+}$ nous avons $<s_{3},s_{1}^{s_{3}}>=<s_{1},s_{3}>$ et $C(s_{2},s_{1}^{s_{3}})=\alpha+\beta\gamma+(\alpha l+\beta m)=4-2=2$ donc $s_{2}s_{1}^{s_{3}}$ est d'ordre $4$ et nous avons $D_{1}^{+}=<s_{1}^{s_{3}},s_{2},s_{3}>$ avec le système de paramètres $\mathcal{P}(1,2,2-\sqrt{2};-\frac{\alpha+\beta m}{m},-\frac{2\beta m}{\alpha+\beta m})$ par un calcul simple. Donc $D_{1}^{+}\simeq D_{3}^{-}$ et comme $D_{1}^{+}\simeq D_{1}^{-}$ nous obtenons $D_{3}^{+}\simeq D_{3}^{-}$.\\
Partant de $D_{3}^{\epsilon}$ nous avons $C(s_{1},s_{2}(s_{2}s_{3})^{4})=4-\alpha=3$ et $C(s_{1},s_{3}(s_{2}s_{3})^{4})=4-\beta=2$ et nous obtenons $D_{3}^{\epsilon}\simeq D_{4}^{\epsilon}$, d'où le résultat.
\end{proof}
\begin{corollary}
Des ``présentations'' de $G$ sont:
\begin{itemize}
\item $(w(8,8,3),(s_{2}s_{3}^{s_{1}})^{3}=1,(s_{2}s_{1}^{s_{3}})^{4}=1)$;\\
\item $(w(8,8,6),(s_{1}s_{2}^{s_{3}})^{4}=1,(s_{2}s_{3}^{s_{1}s_{3}})^{3}=1)$;\\
 \item $(w(3,4,8),(s_{1}s_{2}^{s_{3}s_{2}})^{3}=1)$;\\
\item $(w(4,6,8),(s_{1}s_{3}^{s_{2}})^{3})=1)$.
\end{itemize}
\end{corollary}
\begin{proof}
C'est clair d'après les propositions précédentes.
\end{proof}
\subsubsection{Le groupe $G_{13}$}
\begin{theorem}
Le groupe $G'=G/N$ est isomorphe au groupe de réflexion complexe $G_{13}$.
\end{theorem}
\begin{proof}
Nous partons du groupe $D_{3}^{+}$ qui est isomorphe à $G$. Dans $G'$ il n'y a qu'un seul élément d'ordre $2$. Comme $(s_{1}'s_{3}')^{2}\in D(<s_{1}',s_{3}'>\subset D(G')$ et $(s_{2}'s_{3}')^{4}\in D(<s_{2}',s_{3}'>\subset D(G')$ , on a $z=(s_{1}'s_{3}')^{2}=(s_{2}'s_{3}')^{4}$ et comme $z$ commute à $s_{1}'$, $s_{2}'$ et $s_{3}'$, $z$ est central dans $G'$ et est d'ordre $2$. maintenant $\bar{G}:=G'/<z>=<s_{1}',s_{2}',s_{3}'>\simeq W(B_{3})$ et on a le résultat car $z$ est dans $G'$.
\end{proof}
\begin{theorem}
Des présentations du groupe $G_{13}$ comme quotient d'un groupe de Coxeter $W(p,q,r)$ sont les suivantes:
\begin{enumerate}
  \item $(w(8,8,3),(s_{1}s_{2}^{s_{3}})^{2}=(s_{1}s_{3})^{4})$;
  \item $(w(8,8,6),(s_{2}s_{1}^{s_{3}s_{1}})^{2}=(s_{2}s_{3})^{3}=(s_{1}s_{3})^{4})$;
  \item $(w(3,4,8),(s_{1}s_{3})^{2}=(s_{2}s_{3})^{4})$;
  \item $(w(4,6,8),(s_{1}s_{2})^{2}=(s_{1}s_{3})^{3}=(s_{2}s_{3})^{4})$.
\end{enumerate}
\end{theorem}
\begin{proof}
Nous ne démontrons que la première présentation car les autres sont semblables.\\
Nous nous pla\c  cons dans $G'=G/N$ et appelons $g'$ l'image de $g$ dans $G'$. Soit $z=(s_{1}'s_{2}'^{s_{3}'})^{2}=(s_{1}'s_{3}')^{4}$. Alors $z^{2}=1$ et $z$ est central dans $G'$. Il est clair que $z$ est dans le groupe dérivé de $G'$. $G'/<z>$ est engendré par les images de $s_{1}'$, $s_{2}'^{s_{3'}}$ et $s_{3}'$ avec les relations $((s_{1}'s_{3}')^{4}=1=(s_{1}'s_{2}'^{s_{3}'})^{2}=(s_{3}'s_{2}'^{s_{3}'})^{3}$ donc $G'/<z>$ est isomorphe à $W(B_{3})$ et $G'$ est isomorphe à $G_{13}$.
\end{proof}
\begin{proposition}
Le groupe $G$ possède les $R$-sous-groupes suivants:
\begin{enumerate}
  \item $W(\tilde{G}_{2})$, groupe de Weyl affine de type $G_{2}$.
  \item Le groupe affine associé à $G_{12}$ (ceci implique que dans $G_{13}$ il y a deux classes de conjugaison de réflexions car $G_{12}$ est d'indice $2$ dans $G_{13}$.)
  \item $W(\tilde{B}_{2})$, groupe de Weyl affine de type $B_{2}$.
  \item $W(8,8,2)$ isomorphe au groupe de réflexion complexe affine associé au groupe diédral $D_{8}$.
  \item $W(4,4,8)$ isomorphe au groupe de réflexion complexe affine associé à $G(8,4,2)$.
\end{enumerate}
\end{proposition}
\begin{proof}
On considère le groupe $G$ comme quotient du groupe de Coxeter $W(3,4,8)$. On a:
\begin{itemize}
  \item $C(s_{1},s_{2}^{s_{3}})=3$, donc $s_{1}s_{2}^{s_{3}}$ est d'ordre $6$;
  \item $C(s_{1},s_{3}^{s_{2}})=4-\gamma$ donc $s_{1}s_{3}^{s_{2}}$ est d'ordre $8$;
  \item $C(s_{1},s_{2}^{s_{3}s_{2}})=1$, donc $s_{1}s_{2}^{s_{3}s_{2}}$ est d'ordre $3$;
  \item $C(s_{1},s_{3}^{s_{2}s_{3}})=\gamma$ donc $s_{1}s_{3}^{s_{2}s_{3}}$ est d'ordre $8$;
  \item $C(s_{1},s_{2}^{s_{3}s_{2}s_{3}})=3$, donc $s_{1}s_{2}^{s_{3}s_{2}s_{3}}$ est d'ordre $6$;
  \item $C(s_{1},s_{3}^{s_{2}s_{3}s_{2}})=2$ donc $s_{1}s_{3}^{s_{2}s_{3}s_{2}}$ est d'ordre $4$.
\end{itemize}
Ceci implique par exemple que $<s_{1},s_{2},s_{2}^{s_{3}}>$ est isomorphe au groupe de réflexion complexe associé à $G_{12}$. on procède de la même manière dans les autres cas.
\end{proof}
\subsubsection{Le groupe $N$ et la suite $(\star\star)$}
Le corps $K=\mathbb{Q}(\sqrt{2},i)$ est isomorphe au corps $\mathbb{Q}(\zeta)$ où $\zeta$ est une racine primitive $8$-ième de l'unité. Dans la suite on prend $\zeta=\frac{1+i}{\sqrt{2}}$. On a les relations:\\ $\zeta^{2}=i$, $\zeta^{3}=\frac{-1+i}{\sqrt{2}}$, $\zeta^{4}=-1$, $\zeta^{5}=-\zeta$, $\zeta^{6}=-i=-\zeta^{2}$, $\zeta^{7}=\zeta^{-1}=-\zeta^{3}$.

L'anneau des entiers de $K$ (resp. de $\mathbb{Q}(\zeta)$) est l'anneau $\mathcal{O}(K)=\mathbb{Z}[\sqrt{2},i]$ (resp. $\mathbb{Z}[\zeta]$). Ces deux anneaux sont des anneaux principaux. L'idéal $(1+\zeta)$ de $\mathbb{Z}[\zeta]$ est maximal et l'on a $1-\zeta=-\zeta(1-\zeta+\zeta^{2})(1+\zeta)$ avec $(1-\zeta+\zeta^{2})^{-1}=1-\zeta^{2}-\zeta^{3}$, les idéaux $(1-\zeta)$ et $(1+\zeta)$ sont égaux.

On étudie maintenant le groupe $N$ lorsque $G$ est isomorphe au quotient affine du groupe de Coxeter $W(3,4,8)$.
\begin{theorem}
\begin{enumerate}
  \item Le groupe $N$ est un groupe libre de rang $2$ en tant que $\mathbb{Z}[\zeta]$-module. C'est aussi un groupe commutatif libre de rang $8$.
  \item La suite $(\star\star)$ est non scindée.
\end{enumerate}
\end{theorem}
\begin{proof}
1) On a $\alpha=1$, $\beta=2$, $\gamma=2+\sqrt{2}$, $\Delta=0$ donc $l+2m=2-2\gamma$ et nous choisissons, sans perte de généralité, $l=1+i-\sqrt{2}$ et $2m=1-i-\sqrt{2}$.

Si $n\in N$, alors $\alpha n$, $\beta n$, $\gamma n$, $\alpha ln$, $\beta mn$ et $\theta n$ sont dans $N$, donc $\sqrt{2}n\in N$, $(1+i-\sqrt{2})n\in N$, d'où $in\in N$. Il en résulte que pour $1\leqslant j \leqslant3$, $I_{j}$ est stable par multiplication par $\sqrt{2}$ et $i$. Comme $\alpha=1$, $I_{2}=I_{1}$ et aussi $2I_{1}\subset I_{3}\subset I_{1}$, $lI_{2}\subset I_{3}$ et $mI_{3}\subset I_{2}$.

Nous exprimons maintenant tous les éléments qui interviendront dans la suite en  fonction  de $\zeta$.
On a les formules suivantes:\\
\begin{lemma}
\begin{itemize}
  \item $i\sqrt{2}=(1+\zeta)^{2}(1-\zeta+\zeta^{2})$, $\sqrt{2}=-\zeta^{2}(1+\zeta)^{2}(1-\zeta+\zeta^{2})$;
  \item $1+i=-\zeta^{3}(1+\zeta)^{2}(1-\zeta+\zeta^{2})$, $1-i=-\zeta(1+\zeta)^{2}(1-\zeta+\zeta^{2})$; 
  \item $2=-(1+\zeta)^{4}(1-\zeta+\zeta^{2})^{2}$;
  \item $\gamma=2+\sqrt{2}=-\zeta^{3}(1+\zeta)^{2}$, $4-\gamma=2-\sqrt{2}=-\zeta(1+\zeta)^{2}(1-\zeta+\zeta^{2})$;
  \item $l=-\zeta^{3}(1+\zeta)^{3}(1-\zeta+\zeta^{2})^{2}$, $2m=\zeta^{2}(1+\zeta)^{3}(1-\zeta+\zeta^{2})^{2}$;
  \item $m=-\frac{\zeta^{2}}{1+\zeta}$;
  \item $l+2=-(1+\zeta)^{3}(1-\zeta+\zeta^{2})^{2}(1+\zeta+\zeta^{3})$;
  \item $m\sqrt{2}=-(1+\zeta)(1-\zeta+\zeta^{2}$, $\frac{l}{\sqrt{2}}=\zeta-1=\zeta(1+\zeta)(1-\zeta+\zeta^{2})$.
\end{itemize}
\end{lemma}
\begin{proof}
Par le calcul.
\end{proof}
Nous trouvons maintenant des éléments de $N$.

On a $t_{4}^{2}=(s_{1}(s_{2}s_{3})^{4})^{2}=\frac{-2}{4-\gamma}c_{1}=-\gamma c_{1}$, donc $\gamma\in I_{1}$. Comme $I_{1}$ est stable par multiplication par $4-\gamma$, on voit que $2\in I_{1}$, donc aussi $\sqrt{2}\in I_{1}$.

On a $x_{2}^{2}=(s_{2}(s_{3}s_{1})^{2})^{2}=\frac{-2}{l+2}c_{2}$, donc $\frac{2}{l+2}\in I_{2}=I_{1}$, mais $\frac{2}{l+2}=\frac{2(m+1)}{(l+2)(m+1)}=\frac{2(m+1)}{4-\gamma}=\gamma(m+1)$ donc $\gamma m\in I_{1}$. Il en résulte que $(4-\gamma)\gamma m=2m\in I_{1}$, d'où $1-\sqrt{2}-i\in I_{1}$ et comme $\sqrt{2}\in I_{1}$, on voit que $1-i\in I_{1}$ d'où aussi $1+i\in I_{1}$, ce qui implique que $l\in I_{1}$. On a aussi $\gamma m=2m+m\sqrt{2}\in I_{1}$, donc $m\sqrt{2}\in I_{1}$ ou encore $\frac{1-i}{\sqrt{2}}-1\in I_{1}$ : $-1+\zeta^{-1}\in I_{1}$.

$l$ et $2i$ sont dans $I_{1}$, donc $l-2i=1-i-\sqrt{2}\in I_{1}$. On vient de voir que $\frac{1-i-\sqrt{2}}{\sqrt{2}}\in I_{1}$, donc $\frac{l}{2i}-i\sqrt{2}\in I_{1}$ et comme $-i\sqrt{2}\in I_{1}$, on obtient $\frac{l}{\sqrt{2}}=\zeta-1\in I_{1}$. Comme $2\in I_{1}$, on voit que $1+\zeta\in I_{1}$.

On vérifie maintenant que si on multiplie chacun des éléments précédents par $\zeta$ on obtient encore des éléments de $I_{1}$: $I_{1}$ est stable par multiplication par $\zeta$: c'est l'idéal  maximal $(1+\zeta)$. Tous les éléments écrits ci-dessous sont dans $I_{1}$:

\begin{multline*}
2\zeta=\sqrt{2}+i\sqrt{2};\zeta\sqrt{2}=1+i;\zeta\gamma=2\zeta+\sqrt{2}\zeta, \zeta(4-\gamma);\\
\zeta(1+i)=i\sqrt{2};\zeta(1-i)=\sqrt{2};\zeta l=((1+i)-\sqrt{2})\zeta; \\
2m\zeta=((1-i)-\sqrt{2})\zeta; m\sqrt{2}\zeta=1-\zeta=\frac{-l}{\sqrt{2}};\\
\zeta(1+\zeta)=2\zeta-\zeta(1-\zeta).
\end{multline*}
Comme $(c_{1},c_{2})$ est un système libre, on voit que $N$ est un $\mathbb{Z}[\zeta]$-module libre de rang $2$ et donc $N$ est un groupe commutatif libre de rang $8$.

2) Si $\lambda\in I_{3}$, alors $\lambda m\in I_{2}=I_{1}$, donc $\lambda m$ est un multiple de $1+\zeta$ et $(1+\zeta)^{2}$ divise $\lambda$. Une condition suffisante pour que la suite $(\star\star)$ soit non scindée est que la relation $(\mathcal{E})$
\[
-1=(4-\gamma)\lambda_{1}+(l+2)\lambda_{2}+(m+2)\lambda_{3}
\]
ne puisse être satisfaite quels que soient $\lambda_{j}$ dans $I_{j}$. D'après ce qui précède $1+\zeta$ divise le membre de droite de $(\mathcal{E})$ mais pas le membre de gauche, d'où le résultat.
	
\end{proof}
\section{Le groupe de réflexion complexe $G_{22}$ et groupes associés}
Dans toute la suite, si $G$ est un groupe de réflexion, on suppose que $\Delta(G)=0$.

On sait que le groupe de réflexion complexe $G_{22}$ est isomorphe à une extension centrale non scindée du groupe $W(H_{3})$ par un groupe d'ordre $2$. Si $\tilde{A}_{5}$ est l'extension centrale universelle du groupe alterné $A_{5}$, alors $G_{22}\simeq C_{4}\star \tilde{A}_{5}$ avec sous-groupe d'ordre $2$ amalgamé.

On sait que le corps de définition de $G_{22}$ comme groupe de réflexion complexe est $K=\mathbb{Q}(\sqrt{5},i)$.\\
 Il en résulte que le corps $K_{0}=\mathbb{Q}(\sqrt{5})$. On a alors $\{p,q,r\}\subset\{3,4,5,6,10\}$ avec l'un de $p$, $q$, ou $r$ égal à $5$ ou $10$.\\
  Nous sommes ainsi amenés à examiner les triples suivants pour $(p,q,r)$:\\
  $(3,3,5)$, $(3,3,10)$, $(3,4,5)$, $(3,4,10)$, $(3,6,5)$, $(3,6,10)$, $(3,5,5)$, $(3,5,10)$, \\ 
  $(3,10,10)$, $(4,4,5)$, $(4,4,10)$, $(4,6,5)$, $(4,6,10)$, $(4,5,5)$, $(4,5,10)$, $(4,10,10)$,\\ 
  $(6,6,5)$, $(6,6,10)$, $(6,5,5)$, $(6,5,10)$,$(6,10,10)$, $(5,5,5)$, $(5,5,10)$, \\
  $(5,10,10)$, $(10,10,10)$.
 les polynômes $v_{5}(X)$ et $v_{10}(X)$ possèdent chacun deux racines; pour les distinguer nous mettrons une étoile sur le $5$ (resp. le $10$)  et que les racines correspondantes sont distinctes. On doit donc ajouter les triples:\\
$(3,5,5^{*})$, $(3,10,10^{*})$, $(3,5,10^{*})$, $(4,5,5^{*})$, $(4,5,10^{*})$, $(4,10,10^{*})$, $(6,5,5^{*})$,\\ $(6,5,10^{*})$, $(6,10,10^{*})$, $(5,5,5^{*})$, $(5,5^{*}10)$, $(5,5,10^{*})$, $(5,10,10^{*})$, $(10,10,10^{*})$.\\
 Il y a donc trente neuf cas à examiner. Parmi ceux-ci, on peut enlever $(4,4,5)$ et $(4,4,10)$ qui correspondent à des groupes de réflexion complexes imprimitifs (voir la section 1), ainsi que $(5,5,5^{*})$ et $(5,10,10^{*})$ qui correspondent à des groupes diédraux déjà étudiés précédemment.

Soit $\epsilon\in\{-1,+1\}$. On pose $:\tau_{\epsilon}=\frac{3+\epsilon\sqrt{5}}{2}$. On obtient ainsi les deux racines de $v_{5}(X)$; celles de $v_{10}(X)$ sont alors $\tau_{\epsilon}+1$, ou $4-\tau_{\epsilon}$. On aura aussi $\epsilon\in\{+,-\}$ lorsque cela ne posera pas de problèmes.
\subsection{différents corps et groupes}
Nous cherchons d'abord tous les triples $(p,q,r)$ pour lesquels $K=\mathbb{Q}(\sqrt{5},i)$. Pour cela nous donnons les corps obtenus en prenant chacun de ces triples ainsi que quelque propriétés des groupes correspondants.
\subsubsection{Le corps $K_{1}=K_{1,\epsilon}=\mathbb{Q}(\sqrt{-\epsilon\sqrt{5}})$}
On obtient ce corps pour les triples suivants: $(3,3,5)$, $(6,6,5)$, $(3,6,10)$.\\
(i) $(p,q,r)=(3,3,5)$ avec $\gamma=\tau_{epsilon}$. Posons $w_{1}=\sqrt[4]{5}$, $w_{2}=iw_{1}$, $w_{3}=-w_{1}$ et $w_{4}=-w_{2}$ et soit $\tilde{K}_{1}$ le corps de décomposition du polynôme $P_{1}(X)=X^{4}-5$ dont les racines sont les $w_{i}$. Soit $A_{1}^{\epsilon}$ le groupe obtenu par la construction fondamentale. nous montrons que $A_{1}^{+}$ et $A_{1}^{-}$ sont deux groupes isomorphes.

Nous avons $\mathcal{P(A_{1}^{\epsilon}})=\mathcal{P}(1,1,\tau_{\epsilon};2-\tau_{\epsilon}+\sqrt{-\epsilon\sqrt{5}},2-\tau_{\epsilon}-\sqrt{-\epsilon\sqrt{5}})$. Posons $l_{1}^{\epsilon}:=2-\tau_{\epsilon}+\sqrt{-\epsilon\sqrt{5}}$ et $m_{1}^{\epsilon}:=2-\tau_{\epsilon}-\sqrt{-\epsilon\sqrt{5}}$. Nous avons $\tau_{-1}:=\frac{3-\sqrt{5}}{2}$, $2-\tau_{-1}=\frac{1+\sqrt{5}}{2}$; $\tau_{1}:=\frac{3+\sqrt{5}}{2}$, $2-\tau_{1}=\frac{1-\sqrt{5}}{2}$ donc 
\begin{align*}
l_{1}^{(-1)} & = &\frac{1+\sqrt{5}}{2}+\sqrt[4]{5} & , & m_{1}^{(-1)} & = &\frac{1+\sqrt{5}}{2}-\sqrt[4]{5}&&&&& \\
l_{1}^{(+1)} & = &\frac{1-\sqrt{5}}{2}+i\sqrt[4]{5} & , &  m_{1}^{(+1)} & = &\frac{1-\sqrt{5}}{2}-i\sqrt[4]{5}&&&&&
\end{align*}
Soit $\theta$ l'automorphisme d'ordre $2$ de $\tilde{K}$ défini par $\theta(w_{1})=w_{3}$, $\theta(w_{2})=w_{4}$. Des calculs simples montrent que $\theta(\sqrt{5})=-\sqrt{5}$ et $\theta(i)=i$. On peut remarquer que $\theta(l_{1}^{(+)})=l_{1}^{(-)}$ et $\theta(m_{1}^{(+)})=m_{1}^{(-)}$.

Soit $\tilde{M}:=M\bigotimes_{K}\tilde{K}$ l'espace vectoriel obtenu par extension des scalaires. Soient $\mathcal{A}=(a_{1},a_{2},a_{3})$ (resp. $\mathcal{B}=(b_{1}
,b_{2},b_{3})$) la base de $\tilde{M}$ adaptée pour la représentation de réflexion de $A_{1}^{+}=<s_{1}s_{2},s_{3}>$ (resp. de $A_{1}^{-}=<t_{1},t_{2},t_{3}>$). Si $\sigma:\tilde{M}\to \tilde{M}$ est l'application $\theta$-semi-linéaire définie par $\sigma(a_{i})=b_{i}$ $(1\leqslant i\leqslant 3)$, on a $\sigma s_{i}\sigma^{-1}=t_{i}$ $(1\leqslant i\leqslant 3)$: les groupes $A_{1}^{+}$ et $A_{1}^{-}$ sont isomorphes.

Un calcul simple montre que $(A_{1}^{\epsilon})'$ est infini.

Soit $z_{1}$ l'unique élément de $\mathcal{Z}'$ qui commute à $s_{2}$ et à $s_{3}$. Alors $z_{1}s_{2}$ et $z_{1}s_{3}$ sont des réflexions, $C(s_{1},z_{1}s_{2})=4-\alpha=3$, $C(s_{1},z_{1}s_{3})=4-\beta=3$ et $C(z_{1}s_{2},z_{1}s_{3})=\gamma$ donc $A_{1}^{\epsilon}$ est isomorphe à un $R$-sous-groupe de $A_{2}^{\epsilon}$.\\
(ii) $(P,Q,R)=(6,6,5)$ avec $\gamma=\tau_{\epsilon}$. Nous gardons les notations précédentes. Soit $A_{2}^{\epsilon}$ le groupe obtenu par la construction fondamentale. Nous avons $\mathcal{P}(A_{2}^{\epsilon})=\mathcal{P}(3,3,\tau_{\epsilon};-2-\tau_{\epsilon}+\sqrt{-\epsilon\tau_{\epsilon}},-2-\tau_{\epsilon}-\sqrt{-\epsilon\tau_{\epsilon}})$. nous obtenons $3l_{2}^{(+1)}=-2-\tau_{1}+w_{2}$, $3m_{2}^{(+1)}=-2-\tau_{1}+w_{4}$ et $3l_{2}^{(-1)}=-2-\tau_{1}+w_{12}$, $3m_{2}^{(-1)}=-2-\tau_{1}+w_{3}$. Soit $\theta\in \mathcal{G}al(\frac{\tilde{K}}{\mathbb{Q}})$ défini par $\theta^{2}=id$, $\theta(w_{2})=w_{1}$, $\theta(w_{4})=w_{3}$ alors $\theta(\sqrt{5})=-\sqrt{5}$ et $\theta(\tau_{1})=\tau_{-1}$, d'où $\theta(l_{2}^{(+1)})=l_{2}^{(-1)}$ et $m_{2}^{(+1)}=m_{2}^{(-1)}$. Soit $\sigma$ l'application $\theta$-semi-linéaire définie comme dans le (i). on obtient $\sigma A_{2}^{+}\sigma^{-1}=A_{2}^{-}$.

Un calcul simple montre que $(A_{2}^{\epsilon})'$ est infini.\\
(iii) $(p,q,r)=(3,6,10)$ avec $\gamma=4-\tau_{\epsilon}$. Nous gardons les notations précédentes. Soit $A_{3}^{\epsilon}$ le groupe obtenu par la construction fondamentale. Nous avons $\mathcal{P}(A_{3}^{\epsilon})=\mathcal{P}(1,3,4-\tau_{\epsilon};-4+\tau_{\epsilon}+\sqrt{-\epsilon\sqrt{5}},-4+\tau_{\epsilon}-\sqrt{-\epsilon\sqrt{5}})$. Nous obtenons $l_{3}^{(+1)}=-4+\tau_{+1}+w_{2}$, $3m_{3}^{(+1)}=-4+\tau_{+1}+w_{4}$ et $l_{3}^{(-1)}=-4+\tau_{-1}+w_{1}$, $3m_{3}^{(-1)}=-4+\tau_{-1}+w_{3}$. Soit $\theta\in \mathcal{G}al(\frac{\tilde{K}}{\mathbb{Q}})$ défini par $\theta^{2}=id$, $\theta(w_{2})=w_{1}$ $\theta(w_{4})=w_{3}$ et $\theta(\tau_{+1})=\tau_{-1}$ d'où $\theta(l_{3}^{(+1)})=l_{3}^{(-1}$ et $\theta(m_{3}^{(+1)})=m_{3}^{(-1}$. Soit $\sigma$ l'application semi-linéaire définie comme dans le (i). On obtient $\sigma A_{3}^{+1}\sigma^{-1}=A_{3}^{-1}$.

Les groupes $A_{2}^{\epsilon}$ et $A_{3}^{\epsilon}$ sont isomorphes. En effet si nous partons de $A_{3}^{\epsilon}$, nous avons $C(s_{2},s_{1}(s_{1}s_{3})^{3})=3$ , $C(s_{2},s_{3}(s_{1}s_{3})^{3})=\tau_{\epsilon}$, $A_{3}^{\epsilon}=<s_{1}(s_{1}s_{3})^{3}),s_{2},s_{1}(s_{3}s_{3})^{3})>$ avec$(p,q,r)=(6,6,5)$.

\textbf{Dans la suite nous ne donnons que le minimum d'information pour les autres cas, à l'exception de ceux donnant le corps $K=\mathbb{Q}(\sqrt{5},i)$}.
\subsubsection{Le corps $K_{2}=K_{2,\epsilon}=\mathbb{Q}(\sqrt{6(1+\epsilon\sqrt{5}}))$}
On obtient ce corps pour les triples suivants: $(5,5,3)$, $(10,10,3)$, $(5,10,6)$.\\
(i) $(p,q,r)=(5,5,3)$. Posons $w_{1}=\sqrt{6(1+\sqrt{5})}$, $w_{2}=-w_{1}$, $w_{3}=\sqrt{6(1-\sqrt{5})}$, $w_{4}=-w_{3}$ où $w_{1}$, $w_{2}$, $w_{3}$ et $w_{4}$ sont les racines du polynôme $P_{2}(X)=X^{4}-12X^{2}-144$ dans une extension $\tilde{K}_{2}$ de $\mathbb{Q}$. Le groupe de Galois $\mathcal{G}(\tilde{K}_{2}/\mathbb{Q})$ est un groupe diédral d'ordre $8$. Soient $B_{1}^{\epsilon}, (\epsilon\in\{+,-\})$ les groupes obtenus grâce à la construction fondamentale. On a:
\[
\mathcal{P}(B_{1}^{+})=\mathcal{P}(\tau,\tau,1;-\sqrt{5}+\frac{w_{3}}{2},-\sqrt{5}-\frac{w_{3}}{2})
\]
et 
\[
\mathcal{P}(B_{1}^{-})=\mathcal{P}(\tau_{-},\tau_{-},1;\sqrt{5}+\frac{w_{1}}{2},\sqrt{5}-\frac{w_{1}}{2})
\]
où $\tau=\frac{3+\sqrt{5}}{2}$ et $\tau_{-}=\frac{3-\sqrt{5}}{2}$.

Soit $\theta\in\mathcal{G}(\tilde{K}_{2}/\mathbb{Q})$ défini par $\theta(w_{1})=w_{3}$ et $\theta^{2}=id$. Alors on a $\theta(\tau l_{+})=\tau_{-}l_{-}$ et $\theta(\tau m_{+})=\tau_{-}m_{-}$; comme dans le 1) si $\sigma$ est l'application $\theta$-semi-linéaire $\sigma:\tilde{M}\to\tilde{M}$, alors $\sigma(B_{1}^{+})=B_{1}^{-}$ et ces groupes sont isomorphes.\\
(ii) -- $(p,q,r)=(10,10,3)$. On a $\alpha=\beta=\frac{5+\epsilon\sqrt{5}}{2}$, $\gamma=1$. Soient $B_{2}^{\epsilon}, (\epsilon\in\{+,-\})$ les groupes obtenus grâce à la construction fondamentale. On a:
\[
\mathcal{P}(B_{2}^{\epsilon})=\mathcal{P}(\alpha,\alpha,1;3-2\alpha+\sqrt{3\alpha-6},3-2\alpha-\sqrt{3\alpha-6})
\]
avec $3\alpha-6=\frac{1}{4}(6(1+\epsilon\sqrt{5}))$.\\
-- $(p,q,r)=(5,10,6)$. On a $\alpha=\frac{3-\epsilon\sqrt{5}}{2}$, $\beta=4-\alpha$, $\gamma=3$. Soient $B_{3}^{\epsilon}, (\epsilon\in\{+,-\})$ les groupes obtenus grâce à la construction fondamentale. On a:
\[
\mathcal{P}(B_{3}^{\epsilon})=\mathcal{P}(\alpha,4-\alpha,3;3+\frac{1}{2}\sqrt{6(1+\epsilon\sqrt{5})},3-\frac{1}{2}\sqrt{6(1+\epsilon\sqrt{5})})
\]
Alors comme dans le 1), on voit que les groupes $B_{2}^{+}$, $B_{2}^{-}$, $B_{3}^{+}$, $B_{3}^{-}$ sont isomorphes et $B_{1}^{\epsilon}$ est un $R$-sous-groupe de ceux-ci.

Si $B$ est l'un de ces groupes, alors $B'$ est infini.
\subsubsection{Le corps $K_{3}=K_{3,\epsilon}=\mathbb{Q}(\sqrt{2(-1+\epsilon\sqrt{5}})$}
On obtient ce corps pour les triples suivants: $(5,5,5)$, $(10,10,5)$.\\
(i) $(p,q,r)=(5,5,5)$. Posons $w_{1}=\sqrt{2(-1+\sqrt{5})}$, $w_{2}=-w_{1}$, $w_{3}=\sqrt{2(-1-\sqrt{5})}$, $w_{4}=-w_{3}$ où $w_{1}$, $w_{2}$, $w_{3}$ et $w_{4}$ sont les racines du polynôme $P_{2}(X)=X^{4}-4X^{2}-16$ dans une extension $\tilde{K}_{3}$ de $\mathbb{Q}$. Le groupe de Galois $\mathcal{G}(\tilde{K}_{3}/\mathbb{Q})$ est un groupe diédral d'ordre $8$. Soient $C_{1}^{\epsilon}, (\epsilon\in\{+,-\})$ les groupes obtenus grâce à la construction fondamentale. On a:
\[
\mathcal{P}(C_{1}^{\epsilon})=\mathcal{P}(\tau_{\epsilon},\tau_{\epsilon},\tau_{\epsilon};4-3\tau_{\epsilon}+\sqrt{5(2-\tau_{\epsilon}}),4-3\tau_{\epsilon}-\sqrt{5(2-\tau_{\epsilon}}))
\]
Les groupes $C_{1}^{+}$ et $C_{1}^{-}$ sont isomorphes.\\
(ii) $(p,q,r)=(10,10,5)$. Soient $C_{2}^{\epsilon}, (\epsilon\in\{+,-\})$ les groupes obtenus grâce à la construction fondamentale. On a un groupe diédral affine si $\alpha\neq\beta$ et aussi si $\alpha=\beta=\tau_{\epsilon}+1$ et $\gamma=\tau_{\epsilon}$. Si $\alpha=\beta=\tau_{\epsilon}+1$ et $\gamma=\tau_{-\epsilon}$ alors 
\[
\mathcal{P}(C_{2}^{\epsilon})=\mathcal{P}(\tau_{\epsilon}+1,\tau_{\epsilon}+1,\tau_{-\epsilon};-1+\sqrt{-1-\tau_{-\epsilon}},-1-\sqrt{-1-\tau_{-\epsilon}}).
\]
Les groupes $C_{2}^{+}$ et $C_{2}^{-}$ sont isomorphes et les groupes $C_{1}^{\epsilon}$ sont isomorphes à des $R$-sous-groupes de ceux-ci. La démonstration est semblable à celle du 1).
\subsubsection{Les corps $K_{4}=K_{4,\epsilon}=\mathbb{Q}(\sqrt{(-1-\epsilon\sqrt{5}})$ et $K_{4}'=K_{4,\epsilon}'=\mathbb{Q}(\sqrt{(1+\epsilon\sqrt{5}})$}
Posons $w_{1}=\sqrt{(-1-\sqrt{5})}$, $w_{2}=-w_{1}$, $w_{3}=\sqrt{(-1+\sqrt{5})}$, $w_{4}=-w_{3}$ où $w_{1}$, $w_{2}$, $w_{3}$ et $w_{4}$ sont les racines du polynôme $P_{4}(X)=X^{4}+2X^{2}-4$ dans une extension $\tilde{K}_{4}$ de $\mathbb{Q}$. Le groupe de Galois $\mathcal{G}(\tilde{K}_{4}/\mathbb{Q})$ est un groupe diédral d'ordre $8$. On peut remarquer que $\mathbb{Q}(\sqrt{5},i)$ est un sous-corps de $\tilde{K}_{4}$ et aussi $K_{4}$ et $K_{4}'$.\\
 
(I) On obtient le  corps $K_{4}$ pour les triples suivants: $(5,5,4)$, $(10,10,4)$, $(5,10,4)$.\\
(i) $(p,q,r)=(5,5,4)$.  Soient $D_{1}^{\epsilon}, (\epsilon\in\{+,-\})$ les groupes obtenus grâce à la construction fondamentale. On a:
\[
\mathcal{P}(D_{1}^{\epsilon})=\mathcal{P}(\tau_{\epsilon},\tau_{\epsilon},2;-1-\epsilon\sqrt{5}+w_{1},-1-\epsilon\sqrt{5}+w_{2}).
\]
Les groupes $D_{1}^{+}$ et $D_{1}^{-}$ sont isomorphes.\\
(ii) $(p,q,r)=(10,10,4)$. Soient $D_{2}^{\epsilon}, (\epsilon\in\{+,-\})$ les groupes obtenus grâce à la construction fondamentale. On a:
\[
\mathcal{P}(D_{2}^{\epsilon})=\mathcal{P}(4-\tau_{\epsilon},4-\tau_{\epsilon},2;-6+2\tau_{\epsilon}+w_{1},-6+2\tau_{\epsilon}+w_{2}).
\]
Les groupes $D_{2}^{+}$ et $D_{2}^{-}$ sont isomorphes.\\
(iii) $(p,q,r)=(5,10,4)$. Soient $D_{3}^{\epsilon}, (\epsilon\in\{+,-\})$ les groupes obtenus grâce à la construction fondamentale. On a:
\[
\mathcal{P}(D_{3}^{\epsilon})=\mathcal{P}(\tau_{\epsilon},4-\tau_{\epsilon},2;-2+w_{1},-2+w_{2}).
\]
Les groupes $D_{i}^{\epsilon}$ $(1\leqslant i \leqslant3),\epsilon\in\{+,-\}$ sont isomorphes et si $L$ est l'un d'entre eux, alors $L'$ est infini.

(II) On obtient le  corps $K_{4}'$ pour les triples suivants: $(10,10,10)$, $(5,5,10)$.\\
(i) $(p,q,r)=(10,10,10)$. Soient $D_{4}^{\epsilon}, (\epsilon\in\{+,-\})$ les groupes obtenus grâce à la construction fondamentale. On a:
\[
\mathcal{P}(D_{4}^{\epsilon})=\mathcal{P}(\tau_{\epsilon}+1,\tau_{\epsilon}+1,4-\tau_{\epsilon};-2-\tau_{\epsilon}+iw_{1},-2-\tau_{\epsilon}+iw_{2}).
\]
(ii) $(p,q,r)=(5,5,10)$. Soient $D_{5}^{\epsilon}, (\epsilon\in\{+,-\})$ les groupes obtenus grâce à la construction fondamentale. On a:
\[
\mathcal{P}(D_{5}^{\epsilon})=\mathcal{P}(\tau_{\epsilon},\tau_{-\epsilon},\tau_{\epsilon}+1;-\tau_{\epsilon}+iw_{1},-\tau_{\epsilon}+iw_{2}).
\]
ou bien:
\[
\mathcal{P}(D_{6}^{\epsilon})=\mathcal{P}(\tau_{-\epsilon},\tau_{-\epsilon},\tau_{-\epsilon}+1;3-3\tau_{-\epsilon}+iw_{1},3-3\tau_{-\epsilon}+iw_{2}).
\]
Les groupes $D_{4}^{\epsilon}$, $D_{5}^{\epsilon}$ et $D_{6}^{\epsilon}$ sont isomorphes entre eux mais ils ne sont pas isomorphes aux groupes $D_{i}^{\epsilon}$ $(1\leqslant i\leqslant3, \epsilon\in \{+,-\})$; de plus les groupes $(D_{n}^{\epsilon})'$ sont infinis $(1\leqslant n \leqslant6,\epsilon\in \{+,-\})$.
\subsubsection{Le corps $K=\mathbb{Q}(\sqrt{5})(=K_{0})$}
On obtient le  corps $K$ pour les triples suivants: $(5,5^{*},6)$, $(10,10^{*},6)$, $(5,10,3)$.\\
(i) $(p,q,r)=(5,5^{*},6)$. Soit $E_{1}$ le groupe obtenu grâce à la construction fondamentale. On a:
\[
\mathcal{P}(E_{1})=\mathcal{P}(\tau,3-\tau,3;-1,-3).
\]
(ii) $(p,q,r)=(10,10^{*},6)$. Soit $E_{2}$ le groupe obtenu grâce à la construction fondamentale. On a:
\[
\mathcal{P}(E_{2})=\mathcal{P}(4-\tau,\tau+1,3;-3,-5).
\]
(iii) $(p,q,r)=(5,10,3)$. Soit $E_{3}$ le groupe obtenu grâce à la construction fondamentale. On a:
\[
\mathcal{P}(E_{3})=\mathcal{P}(\tau,\tau+1,1;2\tau-1,2\tau-3).
\]
Les groupes $E_{i}$, $(1\leqslant i \leqslant 3)$ sont isomorphes entre eux et $E_{i}'$, $(1\leqslant i \leqslant 3)$ est infini. On peut remarquer qu'il n'y a aucune forme bilinéaire non nulle invariante.
\subsubsection{Le corps $K=\mathbb{Q}(\sqrt{5},i)$}
\begin{proposition}
On obtient le  corps $K$ pour les triples suivants (dans tous les cas on appelle $G_{i}$ le groupe obtenu grâce à la construction fondamentale):\\
(1) $(p,q,r)=(3,3,10)$. On a $\mathcal{P}(G_{1})=\mathcal{P}(1,1,\tau_{\epsilon}+1;-1-\tau_{\epsilon}+i,-1-\tau_{\epsilon}-i)$.\\
(2) $(p,q,r)=(6,6,10)$. On a $\mathcal{P}(G_{2})=\mathcal{P}(3,3,\tau_{\epsilon}+1;-3-\tau_{\epsilon}+i,-3-\tau_{\epsilon}-i)$.\\
(3) $(p,q,r)=(3,6,5)$. On a $\mathcal{P}(G_{3})=\mathcal{P}(1,3,\tau_{\epsilon}+1;-\tau_{\epsilon}+i,-\tau_{\epsilon}-i)$.\\
(4) $(p,q,r)=(4,6,10)$. On a\\ $\mathcal{P}(G_{4})=\mathcal{P}(2,3,\tau_{\epsilon}+1;(-\tau_{\epsilon}+2)+i(\tau_{\epsilon}-2),(-\tau_{\epsilon}+2)-i(\tau_{\epsilon}-2))$.\\
(5) $(p,q,r)=(4,3,10)$. On a $\mathcal{P}(G_{5})=\mathcal{P}(2,1,\tau_{\epsilon}+1;-\tau_{\epsilon}-i(\tau_{\epsilon}-2),-\tau_{\epsilon}+i(\tau_{\epsilon}-2))$.\\
(6) $(p,q,r)=(4,3,5)$. On a $\mathcal{P}(G_{6})=\mathcal{P}(2,1,\tau_{\epsilon};(1-\tau_{\epsilon})(1-i),(1-\tau_{\epsilon})(1+i)$.\\
(7) $(p,q,r)=(4,6,5)$. On a\\ $\mathcal{P}(G_{7})=\mathcal{P}(2,3,\tau_{\epsilon};(-\tau_{\epsilon}+1)+i(\tau_{\epsilon}-1),(-\tau_{\epsilon}+1)-i(\tau_{\epsilon}-1))$.\\
(8) $(p,q,r)=(5,5^{*},3)$. On a $\mathcal{P}(G_{8})=\mathcal{P}(\tau_{\epsilon},3-\tau_{\epsilon},1;i,-i)$.\\
(9) $(p,q,r)=(10,10^{*},3)$. On a $\mathcal{P}(G_{9})=\mathcal{P}(\tau_{\epsilon}+1,\tau_{-\epsilon}+1,1;-2+i,-2-i)$.\\
(10) $(p,q,r)=(5,10,6)$. On a $\mathcal{P}(G_{10})=\mathcal{P}(\tau_{\epsilon},\tau_{\epsilon}+1,3;-2\tau_{\epsilon}+i,-2\tau_{\epsilon}-i)$.\\
(11) $(p,q,r)=(5,5,6)$. On a $\mathcal{P}(G_{11})=\mathcal{P}(\tau_{\epsilon},\tau_{\epsilon},3;1-2\tau_{\epsilon}+i(\tau_{\epsilon}-1),1-2\tau_{\epsilon}-i(\tau_{\epsilon}-1))$.\\
(12) $(p,q,r)=(10,10,6)$. On a\\ $\mathcal{P}(G_{12})=\mathcal{P}(\tau_{\epsilon}+1,\tau_{\epsilon}+1,3;-(1+2\tau_{\epsilon})+i(\tau_{\epsilon}-2),-(1+2\tau_{\epsilon})-i(\tau_{\epsilon}-2))$.\\
(13) $(p,q,r)=(5,10,3)$. On a $\mathcal{P}(G_{13})=\mathcal{P}(\tau_{\epsilon},4-\tau_{\epsilon},1;-1+i(\tau_{\epsilon}-1),-1-i(\tau_{\epsilon}-i))$.\\
(14) $(p,q,r)=(5,5^{*},4)$. On a $\mathcal{P}(G_{14})=\mathcal{P}(\tau_{\epsilon},3-\tau_{\epsilon},2;-1+i,-1-i)$.\\
(15) $(p,q,r)=(10,10^{*},4)$. On a $\mathcal{P}(G_{15})=\mathcal{P}(\tau_{\epsilon}+1,4-\tau_{\epsilon},2;-3+i,-3-i)$.\\
(16) $(p,q,r)=(5,10,4)$. On a $\mathcal{P}(G_{16})=\mathcal{P}(\tau_{\epsilon},\tau_{\epsilon}+1,2;1-2\tau_{\epsilon}+i,1-2\tau_{\epsilon}-i)$.\\
(17) $(p,q,r)=(10,10,10)$. On a\\ $\mathcal{P}(G_{17})=\mathcal{P}(\tau_{\epsilon}+1,\tau_{\epsilon}+1,\tau_{\epsilon}+1;1-3\tau_{\epsilon}+i(\tau_{\epsilon}-2),1-3\tau_{\epsilon}-i(\tau_{\epsilon}-2))$.\\
(18) $(p,q,r)=(5,5,10)$. On a $\mathcal{P}(G_{18})=\mathcal{P}(\tau_{\epsilon},\tau_{\epsilon},4-\tau_{\epsilon};-\tau_{\epsilon}+i(\tau_{\epsilon}-1),-\tau_{\epsilon}-i(\tau_{\epsilon}-1))$.\\
Ces dix huit groupes sont isomorphes à un groupes noté $G$.
\end{proposition}
\begin{proof}
considérons d'abord le groupe noté $G_{8}$. Nous avons $\alpha l+\beta m=0$ donc $C(s_{2},s_{1}^{s_{3}})=3$, $C(s_{2},s_{3}^{s_{1}})=2$ et $C(s_{2},s_{1}^{s_{3}s_{1}})=4-\tau_{\epsilon}$. Comme $<s_{2},s_{3}>$ est engendré par deux quelconques de ses éléments d'ordre $2$, nous obtenons
\begin{itemize}
  \item $G_{8}=<s_{1},s_{2},s_{1}^{s_{3}}>$ avec $(p,q,r)=(5,5,6)$, donc $G_{8}\simeq G_{11}$,\\
  \item $G_{8}=<s_{3}^{s_{1}},s_{2},s_{3}>$ avec $(p,q,r)=(4,5,3)$, donc$G_{8}\simeq G_{6}$;
\end{itemize}
de même $G_{8}$ est isomorphe à $G_{7}$, $G_{13}$, $G_{18}$, $G_{10}$, et $G_{16}$. En procédant comme dans le 4.1.1, nous voyons que $G_{8}$ est isomorphe à $G_{12}$ puis à $G_{9}$. De la même manière $G_{8}$ est isomorphe à $G_{14}$, $G_{15}$ et $G_{17}$. Si nous partons de $G_{1}$, nous obtenons $G_{2}$ et $G_{13}$, donc $G_{8}$ est isomorphe à $G_{1}$, et $G_{2}$. Si nous partons de $G_{3}$, avec $<s_{1}(s_{1}s_{3})^{3},s_{2},s_{3}(s_{1}s_{3})^{3}>$, nous obtenons $G_{2}$. Si nous partons de $G_{4}$, nous obtenons $G_{5}$ et $G_{6}$. Il en résulte que pour $1\leqslant i,j \leqslant 18$ les groupes $G_{i}$ et $G_{j}$ sont isomorphes.
\end{proof}
Nous donnons maintenant des ``présentations'' du groupe $G$.
\begin{proposition}
On garde les hypothèses et notations de la proposition précédente. On donne des relations qui permettent d'obtenir chaque $G_{i}$ lorsque l'on applique la construction fondamentale au groupe de Coxeter correspondant.

\begin{itemize}
  \item $G_{1}: W(3,3,10), C(s_{1},s_{2}^{s_{3}s_{2}})=2\Leftrightarrow l+m=4-2\gamma\Leftrightarrow\Delta(G_{1})=0$.
  \item $G_{2}: W(6,6,10), C(s_{1},s_{2}^{s_{3}s_{2}})=2\Leftrightarrow 3(l+m)=4-2\gamma\Leftrightarrow\Delta(G_{2})=0$.
  \item $G_{3}: W(3,6,5), C(s_{1},s_{3}^{s_{2}s_{3}})=2\Leftrightarrow (l+3m)=-2\gamma\Leftrightarrow\Delta(G_{3})=0$.
  \item $G_{4}: W(4,6,10), C(s_{3},s_{1}^{s_{2}})=1\Leftrightarrow (2l+3m)=-2(\gamma+1)\Leftrightarrow\Delta(G_{4})=0$.
  \item $G_{5}: W(4,3,10), C(s_{3},s_{1}^{s_{2}})=3\Leftrightarrow (2l+m)=2(1-\gamma)\Leftrightarrow\Delta(G_{5})=0$.
  \item $G_{6}: W(4,3,5), C(s_{3},s_{1}^{s_{2}})=3\Leftrightarrow (2l+m)=2(1-\gamma)\Leftrightarrow\Delta(G_{6})=0$.
  \item $G_{7}: W(4,6,5), C(s_{3},s_{1}^{s_{2}})=3\Leftrightarrow (2l+3m)=-2(1+\gamma)\Leftrightarrow\Delta(G_{2})=0$.
\end{itemize}
\begin{itemize}
  \item $G_{8}: W(5,5,3), (C(s_{3},s_{1}^{s_{2}})=3\;\text{et} \;C(s_{3},s_{2}^{s_{1}})=2)\Leftrightarrow$\\
  $(\Delta(G_{8})=0\;\text{et}\; \alpha\neq\beta)\Longrightarrow\alpha l+\beta m=0$.
  \item $G_{9}: W(10,10,3), (C(s_{3},s_{1}^{s_{2}})=1\;\text{et} \;C(s_{3},s_{2}^{s_{1}})=2)\Leftrightarrow$\\
  $(\Delta(G_{9})=0\;\text{et}\; \alpha\neq\beta)\Longrightarrow\alpha l+\beta m=-4$.
  \item $G_{10}: W(5,10,6), (C(s_{1},s_{2}^{s_{3}})=3\;\text{et} \;C(s_{1},s_{3}^{s_{2}})=1)\Leftrightarrow$\\
  $(\Delta(G_{10})=0\;\text{et}\; \alpha+1=\beta)\Longrightarrow\alpha l+\beta m=-4\alpha$.
  \item $G_{11}: W(5,5,6), (C(s_{1},s_{2}^{s_{3}})=2\;\text{et} \;C(s_{1},s_{3}^{s_{2}})=2)\Leftrightarrow$\\
  $(\Delta(G_{11})=0\;\text{et}\; \alpha=\beta)\Longrightarrow\alpha l+\beta m=2-4\alpha$.
  \item $G_{12}: W(5,5,3), (C(s_{1},s_{2}^{s_{3}})=2\;\text{et} \;C(s_{1},s_{3}^{s_{2}})=2)\Leftrightarrow$\\
  $(\Delta(G_{12})=0\;\text{et}\; \alpha=\beta)\Longrightarrow\alpha l+\beta m=-2$.
  \item $G_{13}: W(5,10,3), (C(s_{1},s_{2}^{s_{3}})=2\;\text{et} \;C(s_{2},s_{1}^{s_{3}s_{1}})=1)\Leftrightarrow$\\
  $(\Delta(G_{13})=0\;\text{et}\; \beta=4-\alpha)\Longrightarrow\alpha l+\beta m=-2$.
  \item $G_{14}: W(5,5,4), (C(s_{2},s_{3}^{s_{1}})=1\;\text{et} \;C(s_{2},s_{1}^{s_{3}s_{1}})=3)\Leftrightarrow$\\
  $(\Delta(G_{14})=0\;\text{et}\; \beta=3-\alpha)\Longrightarrow\alpha l+\beta m=-2$.
  \item $G_{15}: W(10,10,4), (C(s_{2},s_{3}^{s_{1}})=1\;\text{et} \;C(s_{2},s_{1}^{s_{3}s_{1}})=1)\Leftrightarrow$\\
  $(\Delta(G_{15})=0\;\text{et}\; \beta=5-\alpha)\Longrightarrow\alpha l+\beta m=-6$.
  \item $G_{16}: W(5,10,4), (C(s_{2},s_{3}^{s_{1}s_{3}})=1\;\text{et} \;C(s_{3},s_{1}^{s_{2}s_{1}})=1)\Leftrightarrow$\\
  $(\Delta(G_{16})=0\;\text{et}\; \beta=\alpha+1)\Longrightarrow\alpha l+\beta m=2-4\alpha$.
  \item $G_{17}: W(10,10,10), (C(s_{1},s_{2}^{s_{3}s_{2}})=2\;\text{et} \;C(s_{1},s_{2}^{s_{3}})=3)\Leftrightarrow$\\
  $(\Delta(G_{17})=0\;\text{et}\; \alpha=\beta=\gamma)\Longrightarrow\alpha l+\beta m=8-6\alpha$.
\end{itemize}
\begin{itemize}
  \item $G_{18}: W(5,5,10), (C(s_{1},s_{2}^{s_{3}})=1\;\text{et} \;C(s_{1},s_{2}^{s_{3}s_{2}})=2\;\text{et}\;C(s_{1},s_{3}^{s_{2}s_{3}})=2)\Leftrightarrow$\\
  $(\Delta(G_{18})=0\;\text{et}\;\beta=\alpha\;\text{et}\;\gamma=4-\alpha)\Longrightarrow\alpha l+\beta m=-2\alpha$.
\end{itemize}
\end{proposition}
\begin{proof}
Nous ne montrons le résultat que pour $G_{1}$, $G_{8}$ et $G_{18}$, les autres cas ayant des démonstrations semblables.

1) Pour $G_{1}$, nous avons $\alpha=\beta=1$, $\gamma=lm=\tau_{\epsilon}+1$, $\Delta(G_{1})=4-2\gamma-(l+m)$ donc $\Delta(G_{1})=0\Leftrightarrow l+m=4-2\gamma$. Nous avons $C(s_{1},s_{2}^{s_{3}s_{2}})=\gamma+(\gamma-1)^{2}+(\gamma-1)(l+m)=(\gamma-1)(4+l+m)$ donc
\[
 C(s_{1},s_{2}^{s_{3}s_{2}})=2\Leftrightarrow 4+l+m=2(4-\gamma)\Leftrightarrow l+m=4-2\gamma
 \]
 car $(\gamma -1)(4-\gamma)=1$. Ceci montre le résultat dans ce cas.
 
 2) Pour $G_{8}$, nous avons $\alpha^{2}=3\alpha-1$, $\beta^{2}=3\beta-1$, $\gamma=lm=1$;\\ $C(s_{3},s_{1}^{s_{2}})=3=\alpha+\beta+(\alpha l+\beta m)$, $C(s_{3},s_{2}^{s_{1}})=2=1+\alpha\beta+(\alpha l+\beta m)$\\
 donc $(C(s_{3},s_{1}^{s_{2}})=3\; \text{et}\; C(s_{3},s_{2}^{s_{1}})=2)\Longrightarrow 2=\alpha+\beta-\alpha\beta$. Comme $\alpha$ et $\beta$ sont les racines de $v_{5}(X)$, cette égalité n'est possible que si $\beta\neq\alpha$ donc $\alpha+\beta=3$ et $\alpha l+\beta m=0$. Dans ces conditions $\Delta(G_{8}=8-2(\alpha+\beta)-2\gamma-(\alpha l+\beta m)=0$: $(C(s_{3},s_{1}^{s_{2}})=3\; \text{et}\; C(s_{3},s_{2}^{s_{1}})=2)\Longrightarrow(\Delta(G_{8}=0\;\text{et}\;\alpha\neq\beta)$. La réciproque est claire.
 
 3) Pour $G_{18}$, nous avons $\alpha^{2}=3\alpha-1$, $\beta^{2}=3\beta-1$, $\gamma^{2}=5\gamma-5$. Considérons les trois équations suivantes:
 \begin{align}
\label{}
    C(s_{1},s_{2}^{s_{3}})&=&1&=&\alpha+\beta\gamma+(\alpha l+\beta m)   \\
    C(s_{1},s_{2}^{s_{3}s_{2}})&=& 2&=&\beta\gamma+\alpha(\gamma-1)^{2}+(\alpha l+\beta m)(\gamma-1)\\
    C(s_{1},s_{2}^{s_{3}s_{2}})&=& 2&=&\alpha\gamma+\beta(\gamma-1)^{2}+(\alpha l+\beta m)(\gamma-1)
\end{align}
Considérons (6)-(7):\\ $0=\gamma(\beta-\alpha)-(\gamma-1)^{2}(\beta-\alpha)=(\beta-\alpha)(-\gamma^{2}+3\gamma-1=(\beta-\alpha)(4-2\gamma)$.\\
Comme $\gamma\neq 2$, nous obtenons $\beta=\alpha$.\\
Considérons (6)-(1)($\gamma-1$):
$2-(\gamma-1)=\beta\gamma+\alpha(\gamma-1)^{2}-\alpha(\gamma-1)-\beta\gamma(\gamma-1)$\\
donc comme $\beta=\alpha$, $3-\gamma=\alpha(2-\gamma)$ et comme $(2-\gamma)(3-\gamma)=1$, nous avons $\alpha=(3-\gamma)^{2}=4-\gamma$ et $\gamma=4-\alpha$. dans ces conditions $1=\alpha+\alpha(4-\alpha)+(\alpha l+\beta m)$ et $\alpha l+\beta m=-2\alpha$. Un calcul simple montre qu'alors $\Delta(G_{18})=0$. Réciproquement si $\Delta(G_{18})=0$, $\beta=\alpha$ et $\gamma=4-\alpha$, on obtient sans difficultés les relations (5), (6) et (7).
\end{proof}
\begin{theorem}
On garde les hypothèses et notations précédentes. le groupe $G/N(G)$ est isomorphe au groupe de réflexion complexe $G_{22}$.
\end{theorem}
\begin{proof}
On pose $t:=s_{1}s_{2}s_{3}$ et on appelle $t'$ l'image de $t$ dans $G_{j}'$ $(1\leqslant j \leqslant 18)$. On a:
\begin{itemize}
  \item $(t')^{3}=i id$ si $j\in \{1,2,3,8,9,10,14,15,16\}$\\
  \item $(t')^{5}=\pm i id$ si $j\in \{4,5,6,7,11,12,13,17,18\}$.
 \end{itemize}
 On peut remarquer que dans le premier cas le polynôme caractéristique de $t'$ est $P_{t'}(X)=X^{2}-iX-1$ et dans le second cas $P_{t'}(X)=X^{2}--i(\gamma-3)X-1$ ou $P_{t'}(X)=X^{2}--i(\gamma-2)X-1$.
 
 Maintenant on va montrer le résultat pour le groupe $G_{4}$ qui est un quotient du groupe de Coxeter de rang $3$, $W(4,6,10)$. On sait que $N(G_{4})\neq\{1\}$et l'on pose $G_{4}'=G_{4}/N(G_{4})$. On appelle $g'$ l'image de $g$ de $G_{4}$ dans $G_{4}'$. Comme $G_{4}$ est isomorphe à $G_{1}$, toutes ses réflexions sont conjuguées, donc $D(G_{4})=G_{4}^{+}$ est d'indice $2$ dans $G_{4}$. On a donc aussi $D(G_{4}')=(G_{4}')^{+}$ d'indice $2$ dans $G_{4}'$. Maintenant $G_{4}'$ opère fidèlement sur $M'$ qui est de dimension $2$. Comme $SL_{2}(M')$ possède un unique élément d'ordre $2$ qui est central, on a:
 \[
 (s_{1}'s_{2}')^{2}=(s_{1}'s_{3}')^{3}=(s_{2}'s_{3}')^{5}=z
  \]
  donc $D(G_{4}')$ est isomorphe au groupe binaire de l'icosaèdre. De plus $(t')^{5}=(s_{1}'s_{2}'s_{3}')^{5}=-iid_{M'}$, donc $(t')^{5}$ est central dans $G_{4}'$ et $\det (t')^{5}=-1$: $(t')^{5}\notin D(G_{4}')$. On obtient ainsi le résultat: $G$ est isomorphe au groupe de réflexion complexe $G_{22}$.
  \end{proof}
  \begin{proposition}
Des présentations de $G_{22}$ sont les suivantes (on garde les notations $G_{i}$ $(1\leqslant i \leqslant 18)$ et $t=s_{1}s_{2}s_{3}$):
Nous donnons successivement les $18$ présentations.
\begin{itemize}
  \item $G_{1}$: $(w(3,3,10),(s_{2}s_{3})^{5}=(s_{1}s_{2}^{s_{3}s_{2}})^{2}=t^{6},t^{3}=s_{1}t^{3}s_{1}=s_{3}t^{3}s_{3})$;
  \item $G_{2}$: $(w(6,6,10),(s_{1}s_{2})^{3}=(s_{1}s_{3})^{3}=(s_{2}s_{3})^{5}=t^{6},t^{3}=s_{1}t^{3}s_{1}=s_{3}t^{3}s_{3})$;
  \item $G_{3}$: $(w(3,6,5),(s_{1}s_{3})^{3}=(s_{1}s_{2}^{s_{3}s_{2}})^{2}=t^{6},t^{3}=s_{1}t^{3}s_{1}=s_{3}t^{3}s_{3})$;
  \item $G_{4}$: $(w(4,6,10),(s_{1}s_{2})^{2}=(s_{1}s_{3})^{3}=(s_{2}s_{3})^{5}=t^{10},t^{5}=s_{1}t^{5}s_{1}=s_{3}t^{5}s_{3})$;
  \item $G_{5}$: $(w(4,3,10),(s_{1}s_{2})^{2}=(s_{2}s_{3})^{5}=t^{10},t^{5}=s_{1}t^{5}s_{1}=s_{3}t^{5}s_{3})$;
  \item $G_{6}$: $(w(4,3,5),(s_{1}s_{2})^{2}=(s_{3}s_{1}^{s_{2}})^{3}=t^{10},t^{5}=s_{1}t^{5}s_{1}=s_{3}t^{5}s_{3})$;
  \item $G_{7}$: $(w(4,6,5),(s_{1}s_{2})^{2}=(s_{1}s_{3})^{3}=t^{10},t^{5}=s_{1}t^{5}s_{1}=s_{3}t^{5}s_{3})$;
  \item $G_{8}$: $(w(5,5,3),(s_{3}s_{1}^{s_{2}})^{3}=(s_{3}s_{2}^{s_{1}})^{2}=t^{6},t^{3}=s_{1}t^{3}s_{1}=s_{3}t^{3}s_{3})$;
  \item $G_{9}$: $(w(10,10,3),(s_{3}s_{1}^{s_{2}})^{3}=1,(s_{1}s_{2})^{5}=(s_{1}s_{3})^{5}=t^{6},t^{3}=s_{1}t^{3}s_{1}=s_{3}t^{3}s_{3})$;
  \item $G_{10}$: $(w(5,6,10),(s_{2}s_{3})^{5}=(s_{1}s_{3})^{3}=(s_{3}s_{2}^{s_{1}})^{2}=t^{6},t^{3}=s_{1}t^{3}s_{1}=s_{3}t^{3}s_{3})$;
  \item $G_{11}$: $(w(5,5,6),(s_{1}s_{2}^{s_{3}})^{2}=(s_{1}s_{3}^{s_{2}})^{3}=t^{10},t^{5}=s_{1}t^{5}s_{1}=s_{3}t^{5}s_{3})$;
  \item $G_{12}$: $(w(10,10,6),(s_{1}s_{2})^{5}=(s_{1}s_{3})^{5}=(s_{2}s_{3})^{3}=(s_{1}s_{2}^{s_{3}})^{2}=t^{10},t^{5}=s_{1}t^{5}s_{1}=s_{3}t^{5}s_{3})$;
  \item $G_{13}$: $(w(5,10,3),(s_{1}s_{3})^{5}=(s_{1}s_{2}^{s_{3}})^{2}=t^{10},t^{5}=s_{1}t^{5}s_{1}=s_{3}t^{5}s_{3})$;
  \item $G_{14}$: $(w(5,5,4),(s_{2}s_{3})^{2}=(s_{2}s_{1}^{s_{3}s_{1}})^{3}=t^{6},t^{3}=s_{1}t^{3}s_{1}=s_{3}t^{3}s_{3})$;
  \item $G_{15}$: $(w(10,10,4),(s_{2}s_{1}^{s_{3}})^{3}=1,(s_{1}s_{2})^{5}=(s_{1}s_{3})^{5}=(s_{2}s_{3})^{2}=t^{6},t^{3}=s_{1}t^{3}s_{1}=s_{3}t^{3}s_{3})$;
  \item $G_{16}$: $(w(5,10,4),(s_{2}s_{1}^{s_{3}})^{3}=1,(s_{2}s_{3})^{2}=(s_{1}s_{3})^{5}=t^{6},t^{3}=s_{1}t^{3}s_{1}=s_{3}t^{3}s_{3})$;
  \item $G_{17}$: $(w(10,10,10),(s_{1}s_{2})^{5}=(s_{2}s_{3})^{5}=(s_{3}s_{1})^{5}=(s_{1}s_{2}^{s_{3}})^{3}=(s_{1}s_{2}^{s_{3}s_{2}})^{2}=t^{10},t^{5}=s_{1}t^{5}s_{1}=s_{3}t^{5}s_{3})$;
  \item $G_{18}$: $(w(5,5,10),(s_{2}s_{3})^{5}=(s_{1}s_{2}^{s_{3}s_{2}})^{2}=(s_{1}s_{3}^{s_{2}s_{3}})^{2}=t^{10},t^{5}=s_{1}t^{5}s_{1}=s_{3}t^{5}s_{3})$;
\end{itemize}
\end{proposition}
\begin{proof}
Elle est semblable à celle que l'on obtient pour $G_{4}$ grâce au théorème 7.
\end{proof}
\begin{remark}
Les présentations précédentes ne sont peut-être pas les plus économiques.
\end{remark}
\begin{proposition}
On garde les hypothèses et notations précédentes. Alors:
\begin{enumerate}
  \item $N(G)$ est un $\mathbb{Z}[G]$-module libre de rang $8$.
  \item La suite $(\ast\ast)$ est non scindée.
\end{enumerate}
\end{proposition}
\begin{proof}
Nous allons travailler avec la forme $G_{1}$ de $G$.\\
1) Nous savons que $N(G)$ est stable par multiplication par $\gamma$ et par $\theta=4-\alpha-\beta-\gamma-\alpha l=-i$, donc $N(G)$ est un $\mathbb{Z}[\tau,i]$-module car $\gamma=\tau_{\epsilon}+1$. Comme toutes les réflexions de $G$ sont conjuguées, on obtient $I_{1}=I_{2}=I_{3}(:=I)$ et $I$ est un idéal principal car $\mathbb{Z}[\tau,i]$ est un anneau principal. On a $(s_{1}(s_{2}s_{3})^{5})^{2}=\frac{-2}{4-\gamma}c_{1}\in N(G)$ donc d'après ce qui précède $2\in I$. Il en résulte que comme idéal principal non trivial, $I$ est engendré par $2$ ou par $1+i$. $N(G)$ est un $\mathbb{Z}[\tau,i]$-module de rang $4$ donc un $\mathbb{Z}[G]$-module libre de rang $8$.\\
2) Une condition nécéssaire pour que la suite $(\star\star)$ soit scindée est que la condition $(\mathcal{E})$ soit satisfaite:
\[
(\mathcal{E})\quad -1=(4-\gamma)\lambda_{1}+(l+2)\lambda_{2}+(m+2)\lambda_{3}
\]
où chaque $\lambda_{j}$ est dans $I_{j}$, $1\leqslant i \leqslant 3$. Mais chaque $\lambda_{j}$ est un multiple de $1+i$ qui n'est pas inversible. On a donc une impossibilité et la suite $(\star\star)$ est non scindée.
\end{proof}
\section{Remarques finales et questions}
Soit $W$ un groupe de Coxeter de rang fini, $2$-sphérique et irréductible et soit $R:W\to GL(M)$ une représentation de réflexion obtenue par la construction fondamentale. \\
D'après \cite{Z2} proposition 20 (qui est valable avec les hypothèses précédentes) $\ker R$ est sans torsion.\\
- Question 1: $\ker R$ est-il un groupe libre? Si oui, comment en obtenir une base?

Si $R$ est réductible, on a défini le groupe $N(G)$ dans \cite{Z4} théorème 1.\\
- Question 2: le groupe $N(G)$ est-il toujours non trivial? la condition donnée dans la proposition 3 de \cite{Z4} est suffisante mais pas nécéssaire d'après les résultats de \cite{Z3}.\\
- Question 3: Si $N(G)$ est non trivial, est-il toujours muni d'une structure de $\mathcal{O}K$-module?

Peut-on étendre les résultats de ces travaux aux pseudo-groupes de réflexion?

\end{document}